\documentclass[a4paper,11pt]{amsart}

\usepackage{amsmath, amsthm, amsfonts, amssymb, enumerate}
\usepackage[bookmarks=false]{hyperref}
\numberwithin{equation}{section}

\newtheorem{letterthm}{Theorem}

\newtheorem{thm}{Theorem}[section]
\newtheorem{lem}[thm]{Lemma}

\newtheorem{cor}[thm]{Corollary}
\newtheorem{prop}[thm]{Proposition}

\theoremstyle{definition}
\newtheorem{rem}[thm]{Remark}
\newtheorem*{rem*}{Remark}

\newtheorem{exmp}[thm]{Example}
\newtheorem{ques}{Question}
\newtheorem*{conj}{Conjecture}

\newtheorem*{notation}{Notation}
\newtheorem{defn}[thm]{Definition}

\newtheorem*{defn*}{Definition}

\newcommand{\R}{\mathbf{R}}
\newcommand{\C}{\mathbf{C}}

\newcommand{\N}{\mathbf{N}}

\newcommand{\cC}{\mathcal{C}}
\newcommand{\cH}{\mathcal{H}}
\newcommand{\cK}{\mathcal{K}}
\newcommand{\cL}{\mathcal L}
\newcommand{\cM}{\mathcal{M}}
\newcommand{\cN}{\mathcal{N}}
\newcommand{\cO}{\mathcal{O}}
\newcommand{\cP}{\mathcal{P}}
\newcommand{\cS}{\mathcal{S}}
\newcommand{\cU}{\mathcal{U}}

\newcommand{\cn}{\mathfrak{n}}

\newcommand{\wG}{\widetilde{G}}
\newcommand{\wH}{\widetilde{H}}

\newcommand{\Ad}{\operatorname{Ad}}
\newcommand{\id}{\text{\rm id}}

\newcommand{\Aut}{\operatorname{Aut}}

\newcommand{\supp}{\mathord{\text{\rm supp}}}
\newcommand{\conv}{\mathord{\text{\rm conv}}}
\newcommand{\Tr}{\mathord{\text{\rm Tr}}}

\newcommand{\ot}{\otimes}
\newcommand{\ovt}{\mathbin{\overline{\otimes}}}
\newcommand{\op}{^\text{\rm op}}
\newcommand{\actson}{\curvearrowright}
\newcommand{\ds}{\, \mathord{\text{\rm ds}}}
\newcommand{\dt}{\, \mathord{\text{\rm dt}}}
\renewcommand{\dh}{\, \mathord{\text{\rm dh}}}
\newcommand{\dnu}{\, \mathord{\text{\rm d}}\nu}
\newcommand{\dm}{\, \mathord{\text{\rm dm}}}
\newcommand{\lift}{\mathord{\text{\rm lift}}}

\newcommand{\CQGH}{\C[Q;G,H]}

\newcommand{\VN}{\mathord{\text{\rm vN}}}
\newcommand{\VNQGH}{\VN[Q;G,H]}

\newcommand{\GH}{\lift(G/H)}

\newcommand{\OVW}{operator valued weight}
\newcommand{\PO}{properly outer}
\newcommand{\POR}{properly outer relative}
\newcommand{\TFAT}{The following assertions are true.}
\newcommand{\VNA}{von Neumann algebra}

\begin{document}
\title[Crossed-products by locally compact groups]{Crossed-products by locally compact groups:\\ Intermediate subfactors}
\author{R\'emi Boutonnet \& Arnaud Brothier}
\thanks{The first author was partially supported by a PEPS grant from INSMI and the second author was partially supported by NSF Grant DMS-1362138}

\address{R\'emi Boutonnet\\Institut de Math\'ematiques de Bordeaux, Universit\'e de Bordeaux, 351 cours de la Lib\'eration, 33 405 Talence Cedex, France}
\email{remi.boutonnet@math.u-bordeaux.fr}

\address{Arnaud Brothier\\Department of Mathematics, University of Rome Tor Vergata, Via della Ricerca Scientifica 00133 Roma, Italy}
\email{arnaud.brothier@gmail.com\endgraf
https://sites.google.com/site/arnaudbrothier/}

\begin{abstract}
We study actions of locally compact groups on von Neumann factors and the associated crossed-product von Neumann algebras. In the setting of totally disconnected groups we provide sufficient conditions on an action $G \actson Q$ ensuring that the inclusion $Q \subset Q \rtimes G$ is irreducible and that every intermediate subfactor is of the form $Q \rtimes H$ for a closed subgroup $H<G$. This partially generalizes a result of Izumi-Longo-Popa \cite{ILP98} and Choda \cite{Ch78}. We moreover show that one can not hope to use their strategy for non-discrete groups.\end{abstract}

\maketitle

\section{Introduction}

In the theory of von Neumann algebras, the {\it crossed-product construction} associates to an action of a locally compact group $G$ on  a von Neumann algebra $Q$ a new von Neumann algebra, denoted by $Q \rtimes G$, which encodes the action (to some extent). This construction goes back to Murray and von Neumann \cite{MvN36} in the case of state preserving actions of countable groups on abelian von Neumann algebras, and was called the {\it group measure-space construction}. Thus crossed-product algebras appear as one of the most basic examples of von Neumann algebras.
In the case of actions of discrete groups, elementary properties of these crossed-product algebras are quite well understood. For instance:
\begin{enumerate}[(A)]
\item If $Q$ is abelian, $Q = L^\infty(X,\mu)$, and if the corresponding non-singular $G$-action on $(X,\mu)$ is essentially free then $Q' \cap (Q \rtimes G) \subset Q$, so that the crossed-product $Q \rtimes G$ is a factor if and only if the action is ergodic. 
\item If $Q$ is a factor, then $Q' \cap (Q \rtimes G) = \C$ if and only if the action is properly outer, meaning that for all $g \in G$, the corresponding automorphism of $Q$ is not inner. 
\end{enumerate}
Moreover, in both settings one can completely describe all the intermediate subalgebras between $Q$ and $Q \rtimes G$.
Hybrid cases combining aspects of cases (A) and (B) above have been considered recently in \cite{CS16}

In the case of non-discrete groups the picture is not as nice. The main difference is that there is no ``Fourier decomposition'' of elements $Q \rtimes G$. Namely, not every element of $x \in Q \rtimes G$ can be represented as an $L^2$-element $x \in L^2(Q) \ot L^2(G)$. 

Nevertheless, in the setting (A) above of actions on abelian algebras there are some satisfying results. Sauvageot showed in \cite[Section 2]{Sa77} that the equivalence between $Q' \cap (Q \rtimes G) \subset Q$ and the action $G \actson (X,\mu)$ being essentially free still holds. Moreover, in the case of state preserving actions of unimodular groups, a powerful tool is available: the so-called {\it crossed-section equivalence relation}. It relies on the observation that appropriate ${\rm II}_1$ corners of $Q \rtimes G$ can be described by an explicit equivalence relation. We refer to \cite{KPV14} for details and references.

In this article we will be interested in the setting (B) of actions on factors. An action $G \actson Q$ of a locally compact group $G$ on a factor $Q$ is called {\it strictly outer} if $Q' \cap (Q \rtimes G)=\C$. It is known in this case that properly outer actions need not be strictly outer. As we will see, assuming that the action is strictly outer allows to deduce more conclusions. For instance we will see that it implies that the normalizer of $Q$ inside $Q \rtimes G$ is the semi-direct product $\cU(Q) \rtimes G$, see Corollary \ref{normalizer}. In particular we will derive as in the discrete case that for strictly outer actions the pairs $Q \subset Q \rtimes G$ completely characterize the actions up to cocycle conjugacy. We refer to \cite[Theorem 5.1]{V05} for examples of strictly outer actions. We will also give a new criterion providing more examples in Proposition \ref{factoriality}.

Our main goal is to prove an intermediate subfactor theorem. Namely we will provide examples of strictly outer actions of non-discrete groups that satisfy the following property.

\begin{defn*}
We will say that a strictly outer action $G \actson Q$ of a locally compact group on a factor $Q$ satisfies the {\it Intermediate Subfactor Property} if any subfactor of $Q \rtimes G$, containing $Q$ is of the form $Q \rtimes H$ for some closed subgroup $H < G$.
\end{defn*}

In the case of outer actions of discrete groups, Choda \cite{Ch78} proved such an intermediate subfactor property under the extra assumption that the intermediate subfactor is the range of a normal conditional expectation on $Q \rtimes G$. Then Izumi-Longo-Popa \cite{ILP98} were able to show that the existence of such a conditional expectation is automatic for crossed-products by discrete groups. So the result is known for discrete groups. Unfortunately, as we discuss below, there is no hope to adapt this strategy for general locally compact groups. In fact, the general situation is analytically much harder to handle precisely because conditional expectations need not exist in general.

Our main result relies on a very different approach inspired from ${\rm II}_1$-factor techniques. We will restrict our attention to actions of totally disconnected groups. This will allow us to use Fourier decomposition arguments. Here is the statement.

\begin{letterthm}\label{intermediate subfactor}
Consider an action $\sigma: G \actson Q$ of a totally disconnected locally compact group $G$ on a semi-finite factor $Q$. Assume that $\sigma$ is properly outer relative to a compact open subgroup $K_0 < G$ whose action is minimal. 

Then $\sigma$ satisfies the Intermediate Subfactor Property.
\end{letterthm}

We refer to Section \ref{outact} for the definition of a minimal action and to Definition \ref{POR} for the notion of relatively properly outer action. Theorem \ref{intermediate subfactor} applies to Bernoulli shifts $G \actson (\ovt_{G/K} \, Q_0)$ where $Q_0$ is an arbitrary ${\rm II}_1$-factor, and $K$ is any compact open subgroup of $G$, see Section \ref{section:examples}. For instance if $G$ is a closed subgroup of automorphisms of a locally finite tree $T$, and if $G$ acts transitively on $T$, then the Bernoulli shift action $G \actson (\ovt_T \, Q_0)$ satisfies the Intermediate Subfactor Property. We refer to Section \ref{section:examples} for more examples.

As we mentioned above this result is not a simple generalization of the discrete case, even if $Q$ is tracial, because in general an intermediate subalgebra $Q \subset N \subset Q \rtimes G$ does not behave well from the Hilbert theory perspective. Nevertheless we still manage to use Hilbert techniques to perform the proof.

Our approach relies on two ingredients. The first one is an averaging argument. This is where the semi-finiteness assumption on $Q$ will be used. As we will see in Remark \ref{large centralizer}, this technique also allows to deal with some actions on type ${\rm III}$ factors $Q$, but we need to make an assumption that there is a $G$-invariant state on $Q$ that has a {\it large centralizer}.

Our second ingredient is an extension of the notion of {\it support} defined by Eymard \cite{E64}. In the language of quantum groups, the support of an element of $Q \rtimes G$ is {\it the spectrum of the dual action}. In the setting of totally disconnected groups, this notion is particularly well suited, see for instance the proof of Proposition \ref{factoriality}. The key fact that we will use in the proof of Theorem \ref{intermediate subfactor} is that an element whose support is contained in some closed subgroup $H < G$ actually belongs to the subalgebra $Q \rtimes H \subset Q \rtimes G$. This result is certainly known to experts in quantum groups, but as we were not able to find an explicit reference, we will provide a self contained proof in Section \ref{section support}.

In view of Theorem \ref{intermediate subfactor} and Remark \ref{large centralizer} we make the following general conjecture.

\begin{conj}
Any strictly outer action $G \actson Q$ on a factor (of any type) satisfies the Intermediate Subfactor Property.
\end{conj}

\begin{rem*}
The conjecture holds in full generality for compact groups. Indeed, if $K$ is a compact group and $K \actson Q$ is a strictly outer action on an arbitrary von Neumann algebra, then the pair $Q \subset Q \rtimes K$ can be identified with the basic construction of the inclusion $Q^K \subset Q$. This result was showed more generally for {\it integrable} (strictly outer) actions of locally compact groups by Vaes, \cite[Theorem 5.3]{V01}. Moreover, by \cite[Theorem 3.15]{ILP98}, every intermediate subfactor of the inclusion $Q^K \subset Q$ is of the form $Q^L$ form some closed subgroup $L < K$. Combining these two facts indeed yells that $K \actson Q$ satisfies the Intermediate Subfactor Property.
\end{rem*}

Going back to general (not necessarily compact) groups, can one prove the conjecture even when $Q$ is tracial and $G$ is an arbitrary locally compact group?
As we mentioned, Remark \ref{large centralizer} allows to produce actions on type ${\rm III}$ factors. But the condition on large centralizers that we need is never fulfilled in the case of actions on factors of type ${\rm III}_0$. This raises the following question.

\begin{ques}
Can one provide an explicit example of a strictly outer action (of a non-discrete group) on a type ${\rm III}_0$-factor that satisfies the Intermediate Subfactor Property?
\end{ques}

We discuss in Section \ref{sectionCEOVW} the (im)possibility to solve the above conjecture by generalizing the work of Izumi-Longo-Popa. Namely, we completely characterize in the case of arbitrary strictly outer actions $G \actson Q$, existence of normal conditional expectation and operator valued weight for an inclusion $Q \rtimes H \subset Q \rtimes G$, where $H$ is a closed subgroup of $G$.

\begin{letterthm}\label{ovwsg}
Consider a strictly outer action $G \actson Q$ on an arbitrary von Neumann algebra $Q$ and take a closed subgroup $H < G$. Then we have the following characterizations.
\begin{enumerate}
\item There exists a normal faithful semi-finite operator valued weight $T \in \cP(Q \rtimes G,Q \rtimes H)$ if and only if the modular functions $\delta_G$ and $\delta_H$ coincide on $H$.
\item The inclusion $Q \rtimes H \subset Q \rtimes G$ is with expectation if and only if $H$ is open inside $G$.
\end{enumerate}
\end{letterthm}

For both parts in the above theorem, the only if parts follow easily from modular theory. Namely both conditions are easily seen to imply the existence of an operator valued weight and conditional expectation, respectively. Our main contribution is to show that they are actually necessary. 

Note that that for non-strictly outer actions there can exist conditional expectations although $H$ is not open inside $G$. For instance, consider a product $G = G_1 \times G_2$, with the trivial action $G \actson \C$. Assume that $G_2$ is second countable and not discrete. Then $G_1$ is not open inside $G$ but any faithful normal state on $LG_2$ gives rise to a conditional expectation from $LG \simeq LG_1 \ovt LG_2$ onto $LG_1$.

On the other hand we do not know whether the above characterization regarding existence of normal faithful semi-finite operator valued weights holds for arbitrary actions (not necessarily strictly outer).

Next, we will show that nevertheless the strategy of Izumi-Longo-Popa and Choda can be applied to some intermediate subfactors, yielding the following result. We will also mention applications to crossed-products by Hecke pairs of groups, see Corollary \ref{cor:Hecke}.

\begin{letterthm}\label{theo:Heckepair}
Consider an arbitrary strictly outer action $G \actson Q$ and a compact open subgroup $K < G$. Then any subfactor of $Q \rtimes G$ containing $Q \rtimes K$ is of the form $Q \rtimes H$ for some intermediate subgroup $H$; $K < H < G$.
\end{letterthm}

Going back to the conjecture, let us finally mention another partial result. It is about intermediate subfactors that are globally invariant under the dual action $\Delta_Q: Q \rtimes G \to (Q \rtimes G) \ovt LG$. Namely if $Q$ is a factor and if $N$ is an intermediate subfactor such that $\Delta_Q(N) \subset N \ovt LG$, then $N$ is of the form $Q \rtimes H$ for some closed subgroup $H < G$, see for instance \cite[Chapter VII.2]{NT79}.

\begin{ques}
Consider a strictly outer action $\alpha$ of a locally compact quantum group $(M,\Delta)$ on a factor $N$. Can one show that any von Neumann subalgebra $Q \subset M \ltimes_\alpha N$ containing $N$ is globally invariant under the dual action? We refer to \cite{V05} for definitions.
\end{ques}

We mention that the Remark above also applies in the quantum setting. Namely, Vaes' result \cite[Theorem 5.3]{V01} is valid in the general quantum setting, while the correspondence result of Izumi-Longo-Popa \cite[Theorem 3.15]{ILP98} was generalized by Tomatsu to the quantum setting in \cite{To09}.

\subsection*{Acknowledgement}
This work was initiated during the trimester {\it Von Neumann algebras} held at the Hausdorff Institute in Spring 2016. Parts of this work are also due to discussions during the Conference in the memory of Uffe Haagerup held in Copenhagen in June 2016 and during the Oberwolfach meeting on C$^*$-algebras in August 2016. We warmly thank the organizers of all these events. The second author was invited at the Mathematical Institute of Bordeaux in September 2016. He gratefully acknowledges the kind hospitality he received.

Mathematically, we are grateful to Sven Raum for interesting discussions and for pointing out \cite{E64} to us regarding the notion of support. We are also indebted to Amaury Freslon for explaining to us part of the history on quantum groups related to this work and for providing references. Theorem \ref{ovwsg} were motivated by discussions of the first author with Hiroshi Ando and Cyril Houdayer. Many thanks to them. We thank Narutaka Ozawa and Yoshimichi Ueda for helpful comments on a previous version of this preprint.

\tableofcontents

\section{Preliminaries}\label{sec:prelim}

\subsection{General notations}

In all the article, the letter $G$ refers to a locally compact\footnote{We follow the French convention according to which the locally compact assumption also contains the Hausdorff axiom.} group; $m_G$ denotes a Haar measure on $G$, and $\delta_G$ the modular function of $G$. When we consider $L^2$-spaces it is always meant with respect to the measure $m_G$. When integrating functions on $G$, we will sometimes use the notation$\ds$ instead of $\text{\rm d} m_G(s)$, $s \in G$. 
The left regular representation of $G$ on $L^2(G)$ is denoted by $\lambda_G$.

The letter $Q$ refers to an arbitrary von Neumann algebra on which $G$ acts. The action will be denoted by $G \actson Q$, and called generically $\sigma$. By an action, we mean an ultraweakly continuous homomorphism from $G$ into the automorphism group of $Q$.

Given a von Neumann algebra $M$ represented on a Hilbert space $H$, $M'$ denotes its commutant, $(M)_1$ its unit ball in the operator norm, $\cU(M)$ its unitary group and $\Aut(M)$ its automorphism group. We denote by $\cS(M)$ and $\cP(M)$ the set of normal faithful states on $M$ and the set of normal faithful semi-finite weights ({\it nfs} weights for short), respectively. For any weight $\Phi \in \cP(Q)$, consider the left ideal $\cn_\Phi(M) = \{ x\in M \, \vert \, \Phi(x^*x)<\infty\}$, on which $x \mapsto\sqrt{ \Phi(x^*x)}$ defines a norm $\Vert \cdot \Vert_\Phi$. We denote by $L^2(M,\Phi)$ the Hilbert space completion of $\cn_\Phi(M)$ and we write $\Lambda_\Phi$ the inclusion map $\Lambda_\Phi:\cn_\Phi(M) \to L^2(M,\Phi)$.

\subsection{Group actions and crossed-product von Neumann algebras}\label{sec:gp-act}

Let us give the precise definition of our main object of study. We refer to \cite[Chapter X]{Ta03} and \cite{Ha78a,Ha78b} for details about the facts below.

\begin{defn}
Fix an action $G \actson Q$ and represent $Q$ on a Hilbert space $\cH$. 
The {\it crossed-product von Neumann algebra}, denoted by $Q \rtimes G$, is the von Neumann algebra on $L^2(G,\cH)$ generated by the operators $\{\pi(x)\, , \, x \in Q\}$ and $\{u_s\, , \,s \in G\}$ defined by the formulae:
\[(\pi(x) \xi)(t) = \sigma_t^{-1}(x)\xi(t) \qquad \text{and} \qquad (u_s\xi)(t) = \xi(s^{-1}t),\]
for all $\xi \in L^2(G,\cH)$.
\end{defn}

For notational simplicity we will often omit the $\pi$ and identify $\pi(Q)$ with $Q$ in the above definition.

Throughout the article we will always assume that $Q$ is standardly represented on $\cH$, with conjugation operator $J$ and positive cone $\cP$. In this case, we abuse with notations and denote again by $\sigma:g \in G \to \sigma_g \in \cU(\cH)$ the canonical implementation of the action $G \actson Q$, see \cite{Ha75}. Then for all $g \in G$ the operator $\rho_G(g)$ on $L^2(G,\cH)$ defined as follows lies in the commutant of $Q \rtimes G$:
\begin{equation}\label{rhoG} \rho_G(g)(f)(s) = \delta_G(g)^{1/2}\sigma_g(f(sg)), \quad \text{for all } s \in G, f \in L^2(G,\cH).\end{equation}

Denote by $\cK(G,Q)$ the $*$-algebra of compactly supported, $*$-ultrastrongly continuous functions, endowed with product and involution given by the formulae
\[(F_1 \cdot F_2)(t) = \int_G \sigma_s(F_1(ts))F_2(s^{-1}) \ds \qquad \text{and} \qquad F_1^*(t) = \delta_G(t)^{-1}\sigma_{t^{-1}}(F_1(t^{-1})^*),\] 
for all $F_1,F_2 \in \cK(G,Q)$, $t \in G$. The algebra $\cK(G,Q)$ is also a two sided $Q$-module with actions
\[(F \cdot x)(t) = F(t)x \qquad \text{and} \qquad (x \cdot F)(t) = \sigma_t^{-1}(x)F(t),\]
for all $F \in \cK(G,Q)$, $x \in Q$, $t \in G$.
The map $F \mapsto  \int_G u_s F(s) \ds$ defines a $Q$-bimodular embedding of $\cK(G,Q)$ into the crossed-product $Q \rtimes G$.
In this way $\cK(G,Q)$ is viewed as an ultraweakly dense $*$-subalgebra of $Q \rtimes G$.

\subsection{Crossed-products by subgroups}\label{sec:CP}

Given a subgroup $H$ of $G$, one can restrict any action $G \actson Q$ to an action $H \actson Q$. In the case where $H < G$ is closed inside $G$, then the von Neumann subalgebra of $Q \rtimes G$ generated by $Q$ and by the unitaries $u_h$, $h \in H$, is isomorphic to $Q \rtimes H$. This can be seen using induced representations, see \cite[Chapter X.4]{Ta03}. With the same tools one can also compute the commutant of $Q \rtimes H$ inside $B(L^2(G,\cH))$.

\begin{thm}[\cite{NT79}, Theorem VII.1.1]
\label{commutant subgroup}
The commutant of $Q \rtimes H$ inside $B(L^2(G,\cH))$ is the von Neumann algebra generated by the commutant of $Q \rtimes G$ and the subalgebra $\cL^\infty(H\backslash G) \subset L^\infty(G)$ consisting of left-$H$-invariant functions on $G$. Equivalently,
\[Q \rtimes H = (Q \rtimes G) \cap \cL^\infty(H\backslash G)'.\]
\end{thm}

\begin{rem}\label{rkcomm}
Using the $J$-map, the above theorem also allows to compute the basic construction $\langle Q \rtimes G, Q \rtimes H \rangle$ of the inclusion $Q \rtimes H \subset Q \rtimes G$. In fact this description gives the isomorphism
\[\langle Q \rtimes G, Q \rtimes H \rangle \simeq (\cL^\infty(G/H) \ovt Q) \rtimes G,\]
where $G$ acts diagonally.
\end{rem}

\subsection{Modular theory and operator valued weights}

Given a von Neumann algebra $M$, the modular flow of a weight $\Phi \in \cP(M)$ is denoted by $\sigma_t^\Phi \in \Aut(M)$, $t \in \R$. The centralizer of $\Phi$ in $M$ is the subalgebra of elements fixed by the flow $(\sigma_t^\Phi)_t$, it is denoted by $Q^\Phi$. If $\psi \in \cP(M)$ is another weight, $(D\psi:D\Phi)_t$, $t \in \R$, denotes the Connes Radon-Nikodym derivate as defined in \cite[Section 1.2]{Co73}. We will need the following simple Lemma.

\begin{lem}\label{RNautom}
Consider a von Neumann algebra $M$, two weights $\Phi, \Psi \in \cP(M)$, and an automorphism $\theta \in \Aut(M)$. Then $(D\Psi \circ \theta : D\Phi \circ \theta)_t = \theta^{-1}((D\Psi:D\Phi)_t)$, for all $t \in \R$.
\end{lem}
\begin{proof}
Denote by $(e_{i,j})_{1\leq i,j \leq 2}$ the canonical basis of $M_2(\C)$. Put $\tilde M := M \ot M_2(\C)$, and define $\varphi \in \cP(\tilde M)$ by the formula 
\[\varphi \left(\sum_{i,j = 1}^2 x_{i,j} \ot e_{i,j}\right) = \Phi(x_{1,1}) + \Psi(x_{2,2}), \quad \text{for all } x_{i,j} \in M.\]
Note that the weight on $\tilde M$ associated in a similar manner to $\Phi \circ \theta$ and $\Psi \circ \theta$ is $\varphi \circ (\theta \ot \id)$.
By the definition of Connes Radon-Nikodym derivative \cite[Lemme 1.2.2]{Co73}, we have for all $t \in \R$,
\[(D\Psi:D\Phi)_t \ot e_{2,1} = \sigma_t^\varphi(1 \ot e_{2,1}) \quad \text{and} \quad (D\Psi \circ \theta : D\Phi \circ \theta)_t =  \sigma_t^{\varphi \circ (\theta \ot \id)}(1 \ot e_{2,1}).\]
The KMS condition, see \cite[Chapter VIII.1]{Ta03}, implies $\sigma_t^{\varphi \circ (\theta \ot \id)} = (\theta \ot \id)^{-1} \circ \sigma_t^\varphi \circ (\theta \ot \id)$ for all $t \in \R$. The lemma easily follows.
\end{proof}

The notions of modular group and Connes Radon-Nikodym derivative have been defined for normal faithful conditional expectations in \cite{CD75}. We will need the extended definition for operator valued weights defined by Haagerup \cite{Ha79b}.
Let us fix some notations and recall known facts about operator valued weights. We refer to \cite{Ha79a,Ha79b} for precise definitions and proofs.

If $N\subset M$ is an inclusion of von Neumann algebras, then we denote by $\cP(M,N)$ the set of nfs operator valued weights from $M$ to $N$. Given $T \in \cP(M,N)$, we set $\cn_T(M) := \{x \in M \, \vert \, T(x^*x) \in N\}$.
One can define the {\it composition} $\Phi \circ T$ of a weight $T\in\cP(M,N)$, with an nfs weight $\Phi \in \cP(N)$; the resulting weight $\Phi \circ T$ is again normal faithful semi-finite. More generally, it makes sense to compose operator valued weights: if $M_1 \subset M_2 \subset M_3$ are von Neumann algebras and $S \in \cP(M_3,M_2)$, $T \in \cP(M_2,M_1)$ then one can naturally define $T \circ S \in \cP(M_3,M_1)$ in such a way that for all $\Phi \in \cP(M_1)$, we have $\Phi \circ (T \circ S) = (\Phi \circ T) \circ S$.

If $T\in\cP(M,N)$ and $\Phi \in \cP(N)$, then for all $t \in \R$, the modular automorphism $\sigma_t^{\Phi\circ T}$ associated to the weight $\Phi \circ T$ leaves the von Neumann subalgebra $N^c:=N'\cap M$ globally invariant, and its restriction to $N^c$ does not depend on the choice of $\Phi$. This restriction is called the {\it modular flow of $T$}, denoted by $\sigma_t^T$. Moreover, if $S \in \cP(M,N)$ is another operator valued weight then for all $t \in \R$, the Connes Radon-Nikodym derivative $(D(\Phi \circ S):D(\Phi \circ T))_t$ is an element of $N^c$ that does not depend on $\Phi$. It is then denoted by $(DS:DT)_t$ and called the {\it Connes Radon-Nikodym derivative} at time $t$.

If $\cP(M,N)$ is non-empty then $\cP(N',M')$ is non-empty. In particular in this case $\cP(M_1,M)$ is non-empty as well, where $M_1$ denotes the Jones basic construction of the inclusion $N \subset M$, see \cite{Jo83}. Better, if $N \subset M$ is with expectation $E$ there exists $\widehat{E} \in \cP(M_1,M)$ such that $\widehat E(e)= 1$, where $e$ is the Jones projection associated with $E$, see \cite[Lemma 3.1]{Ko85}. $\widehat E$ is called the {\it dual operator valued weight}.

Finally, consider an action $G \actson Q$. By \cite[Theorem 3.1]{Ha78b}, there exists a unique operator valued weight $T_Q\in \cP(Q\rtimes G,Q)$ such that for all $F\in\cK(G,Q)$, $g\in G$ and $x\in Q\rtimes G_+$,
\begin{equation}\label{eqdualweight}
T_Q(F^*F)=(F^*F)(e) \text{ and } T_Q(u_g x u_g^*) = \delta_G(g)u_g T_Q( x ) u_g^*.
\end{equation}
We call it the {\it Plancherel operator valued weight}. For any weight $\Phi \in \cP(Q)$, one defines the {\it dual weight} $\Phi\circ T_Q \in \cP(Q\rtimes G)$. See \cite{Ha78a} for a different construction.

\subsection{Outer actions on von Neumann algebras}\label{outact}

\begin{defn}
We say that a group action $G \actson Q$ on a von Neumann algebra is
\begin{itemize}
\item {\it Properly outer} if no non-trivial element of $G$ acts on $Q$ by inner automorphism. 
\item {\it Strictly outer} if the relative commutant $Q' \cap Q \rtimes G$ is trivial.
\item {\it Minimal} if it is faithful and the fixed point subalgebra $Q^G$ is an irreducible subfactor of $Q$: $(Q^G)' \cap Q = \C$. 
\end{itemize}
\end{defn}

As observed in \cite{V01,V05}, a strictly outer action has to be properly outer, but the converse is not true in general.

\begin{lem}[\cite{V01}, Proposition 6.2]
\label{minaction}
An action of a compact group is strictly outer if and only if it is minimal.
\end{lem} 

Let us record a basic example for later use.

\begin{exmp}
\label{irreducible}
Consider a diffuse factor $Q_0$. For any faithful action of a group $H$ on a finite set $X$, the corresponding Bernoulli action $H \actson Q_0^{\ovt X}$ is strictly outer, hence minimal.
\end{exmp}

In the setting of totally disconnected groups the following lemma will be useful. Its proof is very much inspired from that of \cite[Proposition 6.2]{V01}, with an other input, the Galois correspondence theorem \cite[Theorem 3.15]{ILP98}.

\begin{lem}\label{relcomK2}
Consider a minimal action $K_0 \actson Q$ of a compact group $K_0$. For any open subgroup $K < K_0$, we have $(Q^K)' \cap (Q \rtimes K_0) = LK \subset Q \rtimes K$.
\end{lem}
\begin{proof}
Recall that $Q$ is represented on the Hilbert space $\cH$ and that $Q \rtimes K_0$ is the subalgebra of $B(L^2(K_0)) \ovt Q$ generated by $LK_0 \ot 1$ and by the operators $\pi(a) \in L^\infty(K_0) \ovt Q$, $a \in Q$, defined by $\pi(a)(g) = \sigma_g^{-1}(a)$, $g \in G$. In this picture, we have $\pi(Q^{K_0}) = 1 \ot Q^{K_0}$.

Take $x \in (Q^K)' \cap M$. Since $x$ commutes with $Q^{K_0}$ and $K_0$ acts minimally, we deduce that $x \in B(L^2(K_0)) \ot 1$. Note also that $(Q \rtimes K_0) \cap B(L^2(K_0)) \ot 1 = LK_0 \ot 1$, so we can write $x = z \ot 1$ for some $z \in LK_0$.

For all $a \in Q^K$ and all linear functional $\mu \in Q_*$, we have
\[z(\id \ot \mu)(\pi(a)) = (\id \ot \mu)(\pi(a))z.\]
Note that the functions $(\id \ot \mu)(\pi(a)) \in L^\infty(K_0)$ are left $K$-invariant. Let us prove that these functions generate $\ell^\infty( K\backslash K_0)$ as a von Neumann algebra as $a \in Q^K$ and $\mu \in Q_*$. To see that it is sufficient to check that they separate points of $K\backslash K_0$.
Take $g,h \in K_0$ such that $gh^{-1} \notin K$. Then by \cite[Theorem 3.15]{ILP98}, there exists $a \in Q^K$ such that $\sigma_{gh^{-1}}(a) \neq a$. In particular $\pi(a)(g) = \sigma_g^{-1}(a)$ and $\pi(a)(h) =  \sigma_h^{-1}(a)$ are distinct elements of $Q$. So we can find a linear functional $\mu \in Q_*$ that separates them.

So we arrived at the conclusion that $z$ commutes with $\ell^\infty(K\backslash K_0)$, and in particular with $P(K)$, the orthogonal projection onto $L^2(K)$. This means that $z$ leaves $L^2(K)$ invariant : $z \in LK$.
\end{proof}

\subsection{Fourier algebra, dual action, multipliers}
\label{section:fourier}

For a locally compact group $G$, denote by $A(G)$ its {\it Fourier algebra} as introduced by Eymard \cite{E64}. By definition $A(G)$ is the set of functions on $G$ of the form $\xi \ast \tilde{\eta}$, $\xi,\eta \in L^2(G)$, where $\ast$ denotes the convolution product and $\tilde{\eta}$ is the function $g \mapsto \overline{\eta(g^{-1})}$. Note that for all $g \in G$, we have the equality $(\xi \ast \tilde{\eta})(g) = \langle \lambda(g)\xi,\eta \rangle$. This set is an algebra under the pointwise multiplication.

The norm of $\phi \in A(G)$ is defined to be the minimal value of $\Vert \xi \Vert \Vert \eta \Vert$ as the functions $\xi, \eta \in L^2(G)$ satisfy $\phi = \xi \ast \tilde{\eta}$. With this norm, $A(G)$ is a Banach algebra. The set of compactly supported continuous functions on $G$ is contained (densely) inside $A(G)$. As a Banach space, $A(G)$ is isometric to the predual $(LG)_*$ of $LG$. The duality pairing is given by the well defined formula $< x, \phi > :=  \langle x \xi,\eta\rangle$, for all $x \in LG$, and $\phi = \xi \ast \tilde{\eta} \in A(G)$. We will abuse with notations and write $\phi(x)$ for $x \in LG$ to mean $< x,\phi >$. This notation is somewhat consistent with the fact that $\phi$ is a function on $G$, namely we have: $\phi(u_g) = \phi(g)$ for all $g \in G$.

Using the product on $A(G)$ we can define multipliers on $LG$. More precisely, any $\phi \in A(G)$ gives rise to a normal completely bounded map $m_\phi: LG \to LG$, defined by the formula:
\[< m_\phi(x), \psi> := < x , \psi \phi > \quad \text{ for all } x \in LG, \, \psi \in A(G).\]

More generally, for any action $G \actson Q$, from an element $\phi\in A(G)$ one can construct a multiplier $m_\phi$ on $M:= Q \rtimes G$ in the following way. Consider the unitary operator $W$ on $L^2(G \times G,\cH)$ such that $(W\xi)(g,h) = \xi(g,gh)$ for all $\xi\in L^2(G \times G,\cH)$, $g,h\in G$. Denote by $\Delta:=\Ad(W^*)$ the associated automorphism of $B(L^2(G \times G,\cH))$. We can identify $L^2(G \times G,\cH)$ with $L^2(G,\cH) \otimes L^2(G)$ in such a way that 
\[\Delta(M \otimes 1) \subset M \ovt LG, \text{ and the restriction } \Delta\vert_{L^\infty(G) \otimes 1} = \id.\]

\begin{defn}
With the above notations, the {\it Fourier multiplier} associated with an element $\phi \in A(G)$ is the normal completely bounded map $m_\phi : Q\rtimes G \to Q \rtimes G$ defined by the formula 
\[m_\phi(x) = (\id \ot \phi) \circ \Delta(x \ot 1), \quad x \in Q \rtimes G.\]
\end{defn}

In practice, the multiplier $m_\phi$ is characterized by the formula $m_\phi(au_g) = \phi(g)au_g$ for all $a \in Q$, $g \in G$.
In this way, one easily checks that in the case where $Q = \C$, the two constructions of multipliers coincide.

\section{Support and applications}\label{section support}

In this section we give generalities about the spectrum of the dual action, defined for instance in \cite[Chapter IV.1]{NT79}. We adopt the point of view of Eymard \cite{E64}, and rather talk about ``support'' because we believe it is more transparent for the reader who is familiar with actions of discrete groups, and not so much with the quantum group language. Our goal is to prove Theorem \ref{sgpthm}, regarding elements whose support is contained in a subgroup. This is certainly known to experts but we were not able to find an explicit reference, although it is used in \cite[Theorem VII.2.1]{NT79}. For convenience we tried to keep this section self-contained.

\subsection{Definition and first properties}

Let us fix an arbitrary action $G \actson Q$ on a von Neumann algebra.

\begin{defn}\label{support}
The {\it support} of an element $x \in Q \rtimes G$, denoted by $\supp(x)$, is the set of elements $g \in G$ such that for all $\phi \in A(G)$ satisfying $\phi(g) \neq 0$, we have $m_\phi(x) \neq 0$.
\end{defn}

We also describe the support explicitly in terms of interactions between $Q \rtimes G$ and the copy of $L^\infty(G)$ inside $B(L^2(G,\cH))$. If $\Omega \subset G$ is a measurable subset, we write $P(\Omega) \in L^\infty(G) \ot 1_\cH$ for the orthogonal projection from $L^2(G,\cH)$ onto $L^2(\Omega,\cH)$.

\begin{prop}\label{eqsupp}
Take $x \in Q \rtimes G$ and $g \in G$. The following are equivalent:
\begin{enumerate}[(i)]
\item $g \in \supp(x)$;
\item For any non-empty open set $\Omega \subset G$ we have $P(g\Omega) x P(\Omega)\neq 0$.
\end{enumerate}
\end{prop}
\begin{proof}
(i) $\Rightarrow$ (ii). Consider a non-empty open set $\Omega \subset G$. Take a non-empty open subset $\Omega_0 \subset \Omega$ and an open neighborhood $V$ of $g$ in $G$ such that $V\Omega_0 \subset g\Omega$. Fix a function $\phi \in A(G)$ supported on $V$ such that $\phi(g) \neq 0$.
Then for all $y \in LG$, we have $m_\phi(y)P(\Omega_0) = P(g\Omega)m_\phi(y)P(\Omega_0)$. Indeed the formula is clear for all $y$ of the form $au_h$, $a \in Q$, $h \in G$, and follows for arbitrary $y$ by linearity and density.

Note that $(Q \rtimes G)' L^2(\Omega_0,\cH)$ spans a dense subset of $L^2(G,\cH)$. Indeed, for all $g \in G$, we see that $\rho_G(g)L^2(\Omega_0,\cH) = L^2(\Omega_0g,\cH)$, where $\rho_G(g)$ is defined in \eqref{rhoG}. If follows that $m_\phi(x)P(\Omega_0) \neq 0$.

Extend arbitrarily $\phi \in A(G) = (LG)_*$ to a linear functional $\tilde{\phi}$ on $B(L^2(G))$. With the notations from Section \ref{section:fourier}, since $W$ commutes with $L^\infty(G) \ot 1$, we have
\begin{align*}
(\id \ot \tilde{\phi}) \circ \Ad(W^*)(P(g\Omega)xP(\Omega_0) \ot 1) & = P(g\Omega)(\id \ot \phi) \circ \Ad(W^*)(x \ot 1)P(\Omega_0)\\
& = P(g\Omega)m_\phi(x)P(\Omega_0)\\
& = m_\phi(x)P(\Omega_0) \neq 0.
\end{align*}
We deduce that $P(g\Omega)xP(\Omega_0) \neq 0$, and in particular $P(g\Omega)xP(\Omega) \neq 0$.

(ii) $\Rightarrow$ (i). Fix $\phi \in A(G)$ such that $\phi(g) \neq 0$. Pick an open neighborhood $V$ of $g$ on which $\phi$ does not vanish. Take $\psi \in A(G)$ such that $\psi(h) = 1/\phi(h)$ for all $h \in V$, and put $\phi' := \phi \psi$. In particular we have $\phi'(h) = 1$ for all $h \in V$, and it suffices to show that $m_{\phi'}(x) \neq 0$, since $m_{\phi'} = m_\psi \circ m_\phi$.

Pick a non-empty open set $\Omega \subset G$ such that $g\Omega \cdot \Omega^{-1} \subset V$. We claim that for all $y \in Q \rtimes G$
\[P(g\Omega)m_{\phi'}(y)P(\Omega) = P(g\Omega)yP(\Omega).\]
By linearity and density, it suffices to check this formula for all $y$ of the form $au_h$, $a \in Q$, $h \in G$. 
If $h \in V$, then $m_{\phi'}(au_h) = au_h$ for all $a \in Q$, so the formula is obvious. If $h \notin V$, and $y = au_h$ for some $a \in Q$, then both sides of the formula are equal to $0$. Indeed, since $m_{\phi'}$ is a multiplier, the two terms are scalar multiple of each other, and it suffices to check vanishing of the right-hand side. Since $g \Omega \cdot \Omega^{-1} \subset V$ and $h \notin V$, we deduce that $g\Omega \cap h\Omega = \emptyset$. This leads to
\[P(g\Omega)au_hP(\Omega) = P(g\Omega)P(h\Omega)au_h = 0,\]
as wanted. This proves the claimed equality and hence $m_{\phi'}(y) \neq 0$.
\end{proof}

In the sequel it will sometimes appear that one of the above two descriptions will be better suited to work with than the other. We will freely switch between these two points of view to reach the simplest arguments.

Let us record a few properties of the support.

\begin{lem}
Take $x \in Q \rtimes G$. \TFAT
\begin{enumerate}[a)]
\item The support of $x$ is a closed subset of $G$;
\item If $x$ belongs to $\cK(G,Q)$, then $\supp(x)$ coincides with its support as a function on $G$;
\item If $x = u_g$ for some $g \in G$, then $\supp(x) = \{g\}$.
\end{enumerate}
\end{lem}
\begin{proof}
a) Consider a net $(g_i)_i$ of elements in $\supp(x)$ that converges to some $g\in G$, and take $\phi \in A(G)$ such that $\phi(g) \neq 0$. Since $\phi$ is continuous, for $i$ large enough we also have $\phi(g_i) \neq 0$. Hence $m_\phi(x) \neq 0$, as desired.

b) If $x$ is a continuous function in $\cK(G,Q)$ and $\psi \in A(G)$ observe that $m_\psi(x)$ is the function in $\cK(G,Q)$ defined by $g \mapsto \psi(g)x(g)$. The statement easily follows.

c) If $x = u_g$, for all $\phi \in A(G)$ we have $m_\phi(x) = \phi(g)x$. The result is then obvious.
\end{proof}

\begin{lem}\label{keylemma}
Consider $x\in Q\rtimes G$ and some open subsets $\Omega_1,\Omega_2\subset G$ satisfying the relation $\Omega_1 \cdot \Omega_2^{-1} \cap \supp(x) = \emptyset$. Then 
\[P(\Omega_1) x P(\Omega_2) = 0.\]
In particular, for any open set $\Omega \subset G$, we have $x P(\Omega) = P(\overline{\supp(x)\Omega}) x P(\Omega)$.
\end{lem}
\begin{proof}
We first treat the special case where $\Omega_1 \cdot \overline{\Omega_2^{-1}}\cap \supp(x)=\emptyset$ and $\overline{\Omega_2}$ is compact.

For any $(g,h)\in \Omega_1\times\overline{\Omega_2}$, there exists a non-empty open set $V_{g,h}$ such that
$P(gh^{-1} V_{g,h} ) x P(V_{g,h})=0$. Conjugating this equation with $\rho_G(k)$ for some $k \in G$ we can assume that $V_{g,h}$ is an open neighborhood of $h$. Here, recall that $\rho_G$ is the right action defined in \eqref{rhoG}, which satisfies $\rho_G(k)P(V_{g,h})\rho_G(k^{-1}) = P(V_{g,h} \cdot k^{-1})$.

By compactness, for a fixed $g\in G$ there exists $n \geqslant 1, h_1,\cdots, h_n \in \overline{\Omega_2}$ such that 
$\{V_{g,h_i}\}_{i = 1}^n$ is a finite open cover of $\Omega_2$.
We can then define an open neighborhood $W_g$ of $e \in G$:
\[W_g:= \bigcap_{i=1}^n h_i^{-1} V_{g,h_i}\] 
We get that $P(g W_g) x P(V_{g,h_i})=0$ for any $1\leqslant i\leqslant n, g\in \Omega_1$.
Since $P(\overline{\Omega_2})$ is smaller than the supremum of the projections $\vee_{i=1}^n P(V_{g,h_i})$, we obtain that $P(g W_g) x P(\overline{\Omega_2})=0$ for any $g\in \Omega_1$.
Therefore, $P(\Omega_1) x P(\overline{\Omega_2})=0$ since $\vee_{g\in \Omega_1} P(g W_g) \geqslant P(\Omega_1)$.

Suppose now that $\Omega_2$ is open and non-necessarily relatively compact.
For any $g\in \Omega_2$, there exists a compact neighborhood $K_g$ of $g$ such that $\Omega_1\cdot K_g^{-1}\cap \supp(x)=\emptyset$.
We have that $P(\Omega_1) x P(K_g)=0$ for any $g\in G$, by the proof of above.
This proves the desired equality $P(\Omega_1) x P(\Omega_2) = 0$ because $P(\Omega_2) \leq \vee_{g\in \Omega_2} P(K_g)$.

The second part of the statement follows from taking $\Omega_2 = \Omega$, $\Omega_1 = G \setminus \overline{\supp(x)\Omega}$.
\end{proof}

Before mentioning interesting consequences of this lemma, let us give an essentially equivalent form involving the multipliers.

\begin{lem}\label{trivmult}
Consider $x \in Q \rtimes G$ with compact support, and take a function $\phi \in A(G)$ which is equal to $1$ on a neighborhood of $\supp(x)$. Then $x = m_\phi(x)$.
\end{lem}
\begin{proof}
Take a non-empty open set $\Omega \subset G$ such that $\phi$ is equal to $1$ on $\overline{\supp(x)\Omega} \cdot \Omega^{-1}$.
Proceeding as in the proof of Proposition \ref{eqsupp}, (ii) $\Rightarrow$ (i), one checks that 
\[P(\overline{\supp(x)\Omega})xP(\Omega) = P(\overline{\supp(x)\Omega})m_\phi(x)P(\Omega).\]
Since $\supp(m_\phi(x)) \subset \supp(x)$, Lemma \ref{keylemma} and the above equality imply that \[xP(\Omega) = m_\phi(x)P(\Omega).\]
Moreover, for all $g \in G$, the set $\Omega g$ satisfies the same condition as $\Omega$. So the same equality holds with $\Omega g$ in the place of $\Omega$, for all $g \in G$: $xP(\Omega g) = xP(\Omega g)$.
Since $\Omega$ is open inside $G$, we get that $1 = \vee_g P(\Omega g)$. The equality $x = m_\phi(x)$ follows.
\end{proof}

\begin{cor}\label{cor:prop-supp}
Consider $x,y\in Q\rtimes G$.
\TFAT
\begin{enumerate}[a)]
\item Adjoint: $\supp(x^*) = \supp(x)^{-1}$;
\item Sum: $\supp(x+y)\subset \supp(x)\cup \supp(y)$;
\item Product: $\supp(xy)\subset \overline{\supp(x) \cdot \supp(y)}$. 
\item Vanishing criterion: If $\supp(x)=\emptyset$, then $x=0$.
\end{enumerate}
\end{cor}
\begin{proof}
a) This results from the fact that for all $\phi \in A(G)$, $(m_\phi(x))^* = m_{\phi^*}(x^*)$, where $\phi^* \in A(G)$ is defined by $\phi^*(g) = \overline{\phi(g^{-1})}$, $g \in G$.

b) Consider $g$ in the complementary of $\supp(x)\cup \supp(y)$. Then there exists $\phi_1$ and $\phi_2$ such that $\phi_i(g)\neq 0$, $i = 1,2$, while $m_{\phi_1}(x) = 0$ and $m_{\phi_2}(y) = 0$. In particular the product $\phi := \phi_1 \phi_2$ satisfies $\phi(g) \neq 0$, while $m_\phi(x) = m_{\phi_2} \circ m_{\phi_1}(x) = 0$ and $m_\phi(y) = m_{\phi_1} \circ m_{\phi_2}(y) = 0$. In summary $m_\phi(x+y)=0$ and hence $g \notin \supp(x+y)$.

c) Consider an element $g\in G$ which is not in $\overline{\supp(x) \cdot \supp(y)}$ and take open neighborhoods of the identity $e \in G$, $\Omega$ and $V$, such that $\supp(x) \cdot \supp(y) \cap g\Omega V \Omega^{-1} = \emptyset$.
In particular, for all $h \in \supp(x)^{-1}g\Omega$ and $k \in \supp(y)\Omega$ we have $h^{-1}k \notin V$.
Further, the closures $X := \overline{\supp(x^*)g\Omega}$ and $Y := \overline{\supp(y)\Omega}$ do not intersect. 
Besides, Lemma \ref{keylemma} implies that
\begin{itemize}
\item $yP(\Omega) = P(Y)yP(\Omega)$;
\item $x^*P(g\Omega) = P(X)x^*P(g\Omega)$.
\end{itemize}
Altogether, the following equality shows that $g \notin \supp(xy)$:
\[P(g\Omega)xyP(\Omega) = P(g\Omega)xP(X)P(Y)yP(\Omega) =  P(g\Omega)xP(X \cap Y)yP(\Omega) = 0.\]

d) If $\supp(x)=\emptyset$, then Lemma \ref{keylemma} implies that $x = P(G) x P(G) = 0$.
\end{proof}

\begin{thm}\label{sgpthm}
Consider an action of a locally compact group $G \actson Q$ on an arbitrary von Neumann algebra and take a closed subgroup $H < G$.

Then an element $x \in Q \rtimes G$ belongs to the subalgebra $Q \rtimes H$ if and only if its support is contained in $H$. 
\end{thm}
\begin{proof}
First, assume that $x \in Q \rtimes H$. For all $g \in G \setminus H$ we can find a non-empty open set $\Omega \subset G$ such that $g \Omega \cdot \Omega^{-1} \cap H = \emptyset$. By Lemma \ref{keylemma}, we deduce that $P(g\Omega)xP(\Omega) = 0$, so that $g \notin \supp(x)$. Hence $\supp(x) \subset H$.

Conversely assume that $x \in Q \rtimes G$ is an element with support contained in $H$. In order to show that $x \in Q \rtimes H$, we will use Theorem \ref{commutant subgroup}. This reduces our task to check that $x$ commutes inside $B(L^2(G,\cH))$ with the subalgebra $\cL^\infty(H\backslash G) \subset L^\infty(G)$ consisting of left $H$-invariant functions on $G$. 

Given a left $H$-invariant open set $\Omega$ whose boundary has measure $0$, we have equality of the projections  $P(\Omega) = P(\overline{H\Omega})$. By Lemma \ref{keylemma}, we get $xP(\Omega) = P(\Omega)xP(\Omega)$. Since $x^*$ also has support in $H$ we moreover have $x^*P(\Omega) = P(\Omega)x^*P(\Omega)$. This shows that $x$ commutes with $P(\Omega) \in \cL^\infty(H \backslash G)$. 

{\bf Claim.} The set of functions $P(\Omega)$ with $\Omega$ as above generates $\cL^\infty(H \backslash G)$.

Denote by $q: G \to H \backslash G$ the canonical projection and by $\mu$ a quasi-invariant measure on the coset space $H \backslash G$. The map $f \in L^\infty(H\backslash G,\mu) \mapsto f \circ q \in \cL^\infty(H\backslash G)$ is a normal isomorphism. With this identification, the indicator function $1_U$ of a set $U \subset H \backslash G$ is identified with the indicator function $1_{q^{-1}(U)} = P(q^{-1}(U)) \in \cL^\infty(H\backslash G)$.

Moreover, for all open subset $U \subset H \backslash G$ such that $\mu(\partial U) = 0$, the set $\Omega := q^{-1}(U) \subset G$ is open, $H$-invariant and its boundary has Haar measure $0$. So we are left to check that the span of functions $1_U$ for $U \subset H \backslash G$ such that $\mu(\partial U) = 0$, is ultraweakly dense in $L^\infty(H \backslash G)$. This is a classical fact about Borel measures on locally compact spaces (see e.g. \cite[Proof of (3.33)]{E64} and the references therein).

We deduce from the claim that $x$ commutes with $\cL^\infty(H \backslash G)$. So as wanted, we can use Theorem \ref{commutant subgroup} to deduce that $x \in Q \rtimes H$.
\end{proof}

The special case where $H$ is the trivial subgroup of $G$, the above theorem yields Beurling's Theorem, \cite[Th\'eor\`eme 4.9]{E64}. In particular we stress the following corollary that we will use several times.

\begin{cor}\label{Beurling}
Given an element $x$ in a crossed product von Neumann algebra $Q \rtimes G$, if the support of $x$ is a singleton $\{g\}$, then there exists $a \in Q$ such that $x = au_g$.
\end{cor}
\begin{proof}
If $\supp(x) = \{g\}$, then $\supp(u_g^*x) = \{e\}$. In particular $u_g^*x \in Q$, by Theorem \ref{sgpthm}.
\end{proof}

\subsection{Applications}

Before moving on to the proof of our main theorems, let us mention a few classical results that follow easily from the above properties of the support. 

Our first application concerns a generalization of the so-called Eymard-Steinespring-Tatsuuma's Theorem, \cite[Theorem VII.3.9]{Ta03}.
We start with analyzing how the support behaves under the co-product map.

Given an action $\sigma: G \actson Q$, put $M = Q \rtimes G$ and consider the notations $W \in B(L^2(G\times G,\cH))$ and $\Delta = \Ad(W^*): M \ot 1 \to M \ovt LG$ introduced in Section \ref{section:fourier}. We will view the algebra $M \ovt LG$ as the crossed-product von Neumann algebra associated with the action $\sigma \times \id: G \times G \actson Q$:
\[M \ovt LG \simeq Q \rtimes (G \times G).\]
In this way it makes sense to talk about the support of an element inside $M \ovt LG$.

\begin{lem}\label{Deltasupp}
With the above notations, the following fact hold.
\begin{enumerate}[a)]
\item For all $x \in M$, the support $\Delta(x \ot 1)$ is equal to $\{(g,g) \, , \, g \in \supp(x)\}$;
\item For all $a \in M$ and $b \in LG$, $\supp(a \ot b) =  \supp(a) \times \supp(b)$.
\end{enumerate}
\end{lem}
\begin{proof}
a) For all $\phi, \psi \in A(G)$ the function $\phi \times \psi: (g,h) \mapsto \phi(g)\psi(h)$ belongs to the Fourier algebra $A(G \times G)$ and one easily checks the formula
\begin{equation}\label{eqmult}m_{\phi \times \psi}(\Delta(x \ot 1)) = \Delta(m_{\phi \cdot \psi}(x) \ot 1).\end{equation}
In particular if $g \neq h \in G$, we can find $\phi, \psi \in A(G)$ such that $(\phi \times \psi)(g,h) = \phi(g)\psi(h) \neq 0$, but such that $\phi \cdot \psi = 0$ (just take functions with disjoint supports). By \eqref{eqmult}, we get $m_{\phi \times \psi}(\Delta(x \ot 1)) = 0$ and hence $(g,h)$ does not belong to the support of $\Delta(x \ot 1)$. We thus deduce that the support of $\Delta(x \ot 1)$ is contained in the diagonal of $G \times G$. Further, \eqref{eqmult} easily implies that if $(g,g)$ belongs to the support of $\Delta(x \ot 1)$ then $g \in \supp(x)$.

Conversely, take $g \in \supp(x)$, and $\rho \in A(G \times G)$ such that $\rho(g,g) \neq 0$. We show that $m_\rho(\Delta(x \ot 1)) \neq 0$. Since $\rho$ is continuous, it does not vanish on some open neighborhood $U \times U$ of $(g,g)$. Multiplying $\rho$ by a local inverse if necessary, we may assume that $\rho$ is actually equal to $1$ on $U \times U$. Pick any function $\phi \in A(G)$ supported on $U$ and such that $\phi(g) \neq 0$. We deduce from \eqref{eqmult} that
\[m_{\phi \times \phi}(m_\rho(\Delta(x \ot 1))) = m_{(\phi \times \phi)\cdot \rho}(\Delta(x \ot 1)) =  m_{\phi \times \phi}(\Delta(x \ot 1)) =  \Delta(m_{\phi^2}(x) \ot 1).\]
Since $g \in \supp(x)$ and $\phi^2(g) \neq 0$, the above term is non-zero and hence $m_\rho(\Delta(x \ot 1))) \neq 0$.

b) Take $(g,h) \in \supp(a \ot b)$, and $\phi,\psi \in A(G)$ such that $\phi(g),\psi(h) \neq 0$. Then $(\phi \times \psi)(g,h) \neq 0$, and we get
\[m_\phi(a) \ot m_\psi(b) = m_{\phi \times \psi}(a \ot b) \neq 0.\]
We deduce that $g \in \supp(a)$ and $h \in \supp(b)$. 

The converse inclusion can be treated by using a local inverse as in a). Alternatively, we can use the other description of the support as follows. Take $g \in \supp(a)$, $h \in \supp(b)$ and take a non-empty open set $\Omega \subset G \times G$. We find non-empty open sets $U,V \subset G$ such that $U \times V \subset \Omega$. By assumption we know that $P(gU)aP(U) \neq 0$ and $P(hV)bP(V) \neq 0$, and hence $P(gU \times hV)(a \ot b)P(U \times V) \neq 0$. It then clearly follows that $P((g,h) \cdot \Omega)(a \ot b)P(\Omega) \neq 0$. This proves that $(g,h) \in \supp(a \ot b)$.
\end{proof}

We can now prove the following well known generalization of \cite[Theorem VII.3.9]{Ta03}. This theorem initially deals with the case of group algebras $LG$, that is, with the case of trivial actions $G \actson \C$. The proof given in \cite{Ta03} is rather involved, while an elementary proof relying (implicitly) on the support already appears in \cite[Th\'eor\`eme 3.34]{E64}. In the case of general actions $G \actson Q$, the predual $(Q \rtimes G)_*$ is not identified with an algebra in general, and hence the notion of character doesn't apply anymore. Nevertheless we can still provide an easy proof relying on the notion of support.

\begin{cor}\label{Tak3.9}
Given any action $G \actson Q$, put $M = Q \rtimes G$ and denote by $\Delta : M \ot 1 \to M \ovt LG$ the co-product map defined above. Assume that $x \in M$ is an element such that $\Delta(x \ot 1)$ is of the form $a \ot b$ for some elements $a \in M$ and $b \in LG$.

Then there exists $g \in G$ and $y \in Q$ such that $x = y u_g$.
\end{cor}
\begin{proof} 
We may assume that $x \neq 0$. By Lemma \ref{Deltasupp}, the equality $\Delta(x \ot 1) = a \ot b$ implies that $\supp(a) \times \supp(b)$ is contained in the diagonal of $G \times G$. The only way this can happen is if $\supp(\Delta(x \ot 1))$ is a singleton $\{(g,g)\}$. In this case, $\supp(x) = \{g\}$ and the result follows from Corollary \ref{Beurling}.
\end{proof}

As an immediate corollary, we deduce the following result. 

\begin{cor}\label{normalizer}
Given any strictly outer action $G \actson Q$, the normalizer of $Q$ inside the crossed-product $M := Q \rtimes G$ is equal to $\{a u_g \, , \, g\in G, \, a \in \cU(Q)\}$.

In particular, two strictly outer actions $\alpha_i:G_i\actson Q_i$, $i=1,2$ of locally compact groups are cocycle conjugate if and only if the pairs $Q_1\subset Q_1\rtimes G_1$ and  $Q_2\rtimes Q_2\rtimes G_2$ are isomorphic.
\end{cor}
\begin{proof}
Take $u \in \cN_M(Q)$. Use the notation $\Delta: M \ot 1 \to M \ovt LG$ as above. We have that $\Delta$ is the identity map on $Q \ot 1$. Hence one easily checks that $(u^* \ot 1)\Delta(u \ot 1)$ commutes with $Q \ot 1$. In other words,
\[(u^* \ot 1)\Delta(u \ot 1) \in (Q' \cap M) \ovt LG = 1 \ot LG.\]
So there exists $b \in LG$ such that $\Delta(u \ot 1) = u \ot b$, and we conclude by Corollary \ref{Tak3.9} that $u = au_g$ for some $a \in \cU(Q)$ and $g \in G$.

The second part of the statement is routine. The only if part is always true, even when the actions are not strictly outer, see \cite[Corollary 3.6]{Ta73}. The if part follows from adapting \cite[Proposition 1.7]{BNP} to the case of actions of general locally compact groups. 
\end{proof}

\section{Crossed-products by actions of totally disconnected groups}

\subsection{Notations and tools}

There are several advantages in working with totally disconnected groups. Fix a locally compact group $G$ and an action $G \actson Q$ on a von Neumann algebra and put $M = Q \rtimes G$.

Firstly, given a compact open subgroup $K < G$, one can define a projection $p_K := \mathbf{1}_{K}/m_G(K)$. Note that the net of projections $(p_K)_K$ increases to $1$ as the compact open subgroups $K$ decrease to $\{e\}$. 
Let us mention two elementary properties of these projections.

\begin{lem}\label{Hilbert computations}
Fix a weight $\Phi \in \cP(Q)$ and denote by $\Psi \in \cP(M)$ the corresponding dual weight. Consider a compact open subgroup $K < G$. Then for all $a \in Q^K$ we have $\Vert ap_K \Vert_\Psi = \Vert a \Vert_\Phi/\sqrt{m_G(K)}$.
\end{lem}
\begin{proof}
This is an immediate consequence of \eqref{eqdualweight}.
\end{proof}

\begin{lem}\label{corner}
Given a compact open subgroup $K < G$, the map $a \in Q^K \mapsto ap_K \in p_K(Q \rtimes K)p_K$ is an onto isomorphism of von Neumann algebras.
\end{lem}
\begin{proof}
This is a particular case of Proposition \ref{prop:corner-hecke}.
We give a complete proof of this simpler case for the convenience of the reader.
Since $p_K$ commutes with $Q^K$, it is clear that the map is a normal $*$-morphism. By the computations given in Lemma \ref{Hilbert computations} above, we see that it is moreover injective. To check that it is onto, we only need to prove that it has dense image.

For all $a \in Q$ and $g \in K$, we have $p_K(au_g)p_K = p_Kap_K = E(a)p_K$, where $E:Q\rightarrow Q^K$ is the conditional expectation. 
Hence $p_K(au_g)p_K$ belongs to the range of our map, proving the lemma.
\end{proof}

Secondly, for all open subgroup $K < G$ and any action $G \actson Q$, there always exists a faithful normal conditional expectation $E_K : Q \rtimes G \to Q \rtimes K$ such that $E_K(u_g) = \mathbf 1_{K}(g)$ for all $g \in G$. In the case where $K$ is compact open, one sees that the multiplier $m_\phi$ associated with the function $\phi := \mathbf{1}_K \in A(G)$ gives the desired expectation. In the general case of open subgroups $\phi$ does not necessarily belong to $A(G)$ but it is still positive definite and one can use \cite[Theorem 3.1.a]{Ha78b} to construct the associated multiplier $m_\phi$. 

An alternative way to construct the expectation $E_K$ is by considering modular flows as in \cite{HR16}. With this point of view, it becomes obvious that $E_K$ preserves the Plancherel operator valued weight: $T_Q \circ E_K = T_Q$.

When $G$ is totally disconnected, one checks that the support of an element $x \in M$ is described as follows:
\begin{equation}\label{supptdlc} \supp(x) = \{ g \in G \, : \,  E_K(u_g^*x) \neq 0 \text{ for all  compact open subgroups } K < G\}.\end{equation}

\begin{notation}
Given a compact open subgroup $K$ of a group $G$ and a set $S \subset G/K$, we denote by $\lift(S)$ a set of representatives of $S$ inside $G$. This means that for all $g \in \lift(S)$ we have $gK \in S$ and for all $s \in S$ there exists a unique element $g \in \lift(S)$ such that $gK = s$.
\end{notation}

\begin{lem}\label{compactKdec}
Consider a compact open subgroup $K < G$. Given an element $x \in M$ with compact support, the map $g \mapsto u_gE_K(u_g^*x)$ is right $K$-invariant on $G$ and compactly supported. Moreover,
\[x = \sum_{g \in \lift(G/K)}  u_gE_K(u_g^*x).\]
\end{lem}
\begin{proof}
Fix $S \subset G/K$ a finite set such that the support of $x$ is contained in $\lift(S) \cdot K$.
The function $\phi = \mathbf{1}_{\lift(S)K} \in A(G)$ is equal to $1$ on a neighborhood of the support of $x$, so by Lemma \ref{trivmult}, we have that $x = m_\phi(x)$. Moreover, $\phi$ can be decomposed as $\phi = \sum_{g \in \lift(S)} \mathbf{1}_K(g^{-1}\, \cdot )$. One easily checks that the corresponding multiplier satisfies
\[m_\phi = \sum_{g \in \lift(S)} u_gE_K(u_g^* \,  \cdot).\]
We leave the rest of the proof to the reader.
\end{proof}

Although we will not use this fact, let us mention that for a general element $x \in M$, a $K$-decomposition as above still makes sense. In this case the sum that appears is infinite but it converges in the Bures topology associated with the inclusion $Q \rtimes K \subset Q \rtimes G$, with expectation $E_K$. We refer to \cite[Section 2]{CS16} or to the original book \cite{Bu71} for the definition of the Bures topology.

\subsection{Strictly outer actions}

The following proposition combines the "Fourier coefficient approach" used in the setting of discrete groups with Lemma \ref{minaction} about actions of compact groups.

\begin{prop}\label{factoriality}
Consider a properly outer action of a totally disconnected locally compact group $\sigma: G \actson Q$.
The action $\sigma$ is strictly outer if and only if $G$ admits a compact open subgroup that acts minimally.

In that case any unitary $u \in Q \rtimes G$ normalizing $Q$ is of the form $au_g$ for some $a \in Q$, $g \in G$.
\end{prop}
\begin{proof}
Assume that $\sigma:G\actson Q$ is strictly outer. Then one easily checks from the definitions that $\sigma$ is properly outer. Moreover, the restriction of $\sigma$ to any compact subgroup is strictly outer and thus minimal by \cite[Proposition 6.2]{V01}. 

Conversely, assume that $\sigma:G\actson Q$ is properly outer and admits a compact open subgroup $K_0$ that acts minimally.
Note that any subgroup $K < K_0$ acts minimally as well, and hence in a strictly outer way.

Put $M = Q \rtimes G$ and take a non-zero element $x \in Q' \cap M$. 
By Proposition \ref{Beurling} we only need to show that the support of $x$ is the singleton $\{e\}$. 
The support of $x$ is non-empty by Corollary \ref{cor:prop-supp}.d. Take $g \in \supp(x)$.
For all compact open subgroups $K < K_0$, put $a_K := E_K(u_g^*x) \in Q \rtimes K$. Since $g \in \supp(x)$, these elements $a_K$ are non-zero. 

By minimality, for all $K' < K < K_0$, we have 
\begin{equation}\label{1111} a_{K'}^*a_K \in Q' \cap (Q \rtimes K) = \C.\end{equation}
When $K = K'$ this relation tells us that $a_K$ is a (non-zero) multiple of a unitary element. Then \eqref{1111} further implies that all these unitaries are proportional to each other. In particular, for all $K < K_0$, $a_{K_0}$ is proportional to $a_{K} \in Q \rtimes K$.

We conclude that $a_{K_0}$ is contained in $Q \rtimes K$, for all $K < K_0$. Hence its support is equal to $\{e\}$, implying that $a_{K_0} \in Q$. Moreover, $a_{K_0}$ satisfies
\[a_{K_0}x = \sigma_{g^{-1}}(x)a_{K_0}, \quad \text{ for all } x \in Q.\]
Since the $G$-action is properly outer, this gives $g = e$, as desired.
The statement on the normalizer then follows from Corollary \ref{normalizer} (although this could be checked directly by similar computations on the support).
\end{proof}

\subsection{Intermediate subfactors}

We now turn to the question of determining all intermediate subfactors $Q \subset N \subset Q \rtimes G$. In order to establish our main result Theorem \ref{intermediate subfactor}, we will need to be able to compute relative commutants of the form $(Q^K)' \cap Q \rtimes G$ for small compact open subgroups $K < G$.
This forces us to strengthen our assumptions on the action.

\begin{defn}\label{POR}
Given a subgroup $K$ of $G$ we will say that the action $G \actson Q$ is {\it properly outer relative to $K$} if the following holds. The only elements $g \in G$ for which there exists a non-zero $a \in Q$ such that $\sigma_g(x)a = ax$ for all $x \in Q^K$ are the elements of $K$: $g \in K$.
\end{defn}

Note that an action is \PO\ if and only if it is \POR\ to the trivial subgroup. We will provide examples of relatively properly outer actions in the next section.

\begin{lem}\label{relcomK1}
Consider an action $\sigma: G \actson Q$ of a totally disconnected group $G$ which is properly outer relative to a compact open subgroup $K < G$. Put $M = Q \rtimes G$. We have \[(Q^K)' \cap M \subset Q \rtimes K.\]
\end{lem}
\begin{proof}
Take $x \in (Q^K)' \cap M$. We show that $\supp(x) \subset K$. Fix $g \in \supp(x)$, so that $E_K(u_g^*x) \neq 0$. Since the net $(p_L)_L$ converges strongly to $1$ when $L \searrow \{e\}$, we can find a compact open subgroup $L_0 < G$ such that $p_{L_0}E_K(u_g^*x)p_{L_0} \neq 0$.

Define $L_1 = L_0 \cap K \cap gKg^{-1}$ and $L := \cap_{h \in K} hL_1h^{-1}$. Since $L_1$ is open, it has finite index inside $K$, and hence the intersection defining $L$ is in fact finite. We conclude that $L < K$ is an open subgroup which is moreover contained in $L_0 \cap K \cap gKg^{-1}$. 
In particular $p_L \geq p_{L_0}$, so that $p_LE_K(u_g^*x)p_L \neq 0$. Since $L$ is normal inside $K$, the projection $p_L$ commutes with $u_h$ for all $h \in K$. By Lemma \ref{compactKdec}, we have
\[p_LE_K(u_g^*x)p_L = \sum_{h \in \lift(K/L)} p_Lu_hE_L(u_{gh}^*x)p_L = \sum_{h \in \lift(K/L)} u_hp_LE_L(u_{gh}^*x)p_L.\]
The non-vanishing of this sum implies that there exists $h \in K$ such that $X := p_LE_L(u_{gh}^*x)p_L$ is non-zero.

Since $L < K \cap gKg^{-1}$, we have that $Q^K \subset Q^L$ and $\sigma_g(Q^K) = Q^{gKg^{-1}} \subset Q^L$. Hence for all $b \in Q^K$, the two elements $b$ and $\sigma_g(b)$ commute with $p_L$. It follows
\[X b = p_LE_L(u_{gh}^*xb)p_L = p_LE_L(u_{gh}^*bx)p_L = p_L\sigma_g(b)E_L(u_{gh}^*x)p_L = \sigma_g(b)X.\]

By Lemma \ref{corner}, $p_L(Q \rtimes L)p_L = Q^L p_L$, so there exists a unique (non-zero) $a \in Q^L$ such that $X = ap_L$. By uniqueness, we see that $a$ satisfies $ab = \sigma_g(b)a$ for all $a \in Q^K$. We conclude that $g \in K_0$, and hence $\supp(x) \subset K$. By Theorem \ref{sgpthm} we have that $x \in Q \rtimes K$.
\end{proof}

\begin{prop}\label{relcomK}
Consider an action $\sigma: G \actson Q$ of a totally disconnected group $G$ which is properly outer relative to a compact open subgroup $K_0 < G$ whose action is minimal. For all open subgroup $K < K_0$, we have $(Q^K)' \cap M = LK \subset Q \rtimes K$.
\end{prop}
\begin{proof}
Fix $x \in (Q^K)' \cap M$. By Lemma \ref{relcomK1} applied to $K_0$, we see that $x \in Q \rtimes K_0$. Since $K_0$ acts minimally, the result follows from Lemma \ref{relcomK2}. 
\end{proof}

In order to prove Theorem \ref{intermediate subfactor}, we will use a convex combination argument. The following lemma will be needed.

\begin{lem}[\cite{HI15}, Lemma 4.4]\label{HI15}
Consider a von Neumann algebra $\cM$, a weight $\Theta \in \cP(\cM),$ and a $\sigma$-weakly closed convex subset $\cC$ of $\cM$.

If $\cC$ is bounded both in the operator norm and the  $\Vert \cdot \Vert_\Theta$-norm, then $\Lambda_\Theta(\cC)$ is $\Vert \cdot \Vert_\Theta$-closed in $L^2(\cM,\Theta)$.
\end{lem}

\begin{proof}[Proof of Theorem \ref{intermediate subfactor}]
Fix an intermediate subfactor $Q \subset N \subset M$. Set $H := \{g \in G \, , \, u_g \in N\}$, so that $Q \rtimes H \subset N$.
Take $x_0 \in N$. We will show that the support of $x_0$ is contained in $H$. This will conclude by Theorem \ref{sgpthm}.

Denote by $\Tr$ a normal faithful semi-finite trace on $Q$, and denote by $\Psi \in \cP(M)$ the corresponding dual weight.

{\bf Step 1.} For all compact open subgroups $K < K_0$, all $g \in G$ and all finite trace projection $q \in Q^K$, if there exists $x \in N$ such that $qp_KE_{K}(u_g^*x)qp_{K} \neq 0$, then $u_gqp_K \in (N)_1 p_K$.

Fix $K < K_0$, $q \in Q^K$, $x \in N$ and $g \in G$ such that $qp_KE_{K}(u_g^*x)qp_K \neq 0$.
Put $x' := \sigma_g(q)xq \in N$. We have $p_KE_{K}(u_g^*x')p_{K} = qp_KE_{K}(u_g^*x)qp_{K} \neq 0$.
By Lemma \ref{corner}, there exists a unique $a \in Q^K$ such that $p_KE_{K}(u_g^*x')p_{K} = ap_K$. Note that $a \in qQq$. The element $y := \sigma_g(a^*)x' \in \sigma(q)Nq$ satisfies $p_KE_K(u_g^*y)p_K = a^*ap_K \neq 0$. Lemma \ref{Hilbert computations}, implies that 
\[\Psi(p_Ku_g^*yp_K) = \Psi(p_KE_K(u_g^*y)p_K) = \frac{1}{m_G(K)}\Vert a \Vert_{\Tr}^2 \neq 0.\]

Denote by $\cC \subset N$ the ultraweak closure of the convex hull 
\[\cC_0 := \conv (\{\sigma_g(u)yu^* \, , \, u \in \cU(qQ^{K}q)\}).\]
We proceed as in the proof of \cite[Theorem 4.3 (5) $\Rightarrow$ (6)]{HI15} to show that $\cC p_{K}$ is bounded in $\Vert \cdot \Vert_\Psi$-norm, except that we have slightly weaker assumptions.
Since $Q$ is contained in the centralizer of $\Psi$, it is clear from the triangle inequality that $\Vert Xp_{K} \Vert_\Psi \leq \Vert yp_{K} \Vert_\Psi$ for all $X \in \cC_0$. Now if $X \in \cC$ is arbitrary, take a net $(X_j)_{j \in J} \subset \cC_0$ that converges ultrastrongly to $X$. Then $(X_j^*X_j)_{j \in J}$ converges ultraweakly to $X^*X$, and since $\Psi$ is lower ultraweakly semi-continuous (see \cite[Theorem VII.1.11.(iii)]{Ta03}) we get that 
\[\Vert Xp_K \Vert_\Psi^2 = \Psi(p_KX^*Xp_K) \leq \liminf_j \Psi(p_KX_j^*X_jp_K) \leq \Vert yp_{K} \Vert_\Psi^2.\]
Note moreover that $ \Vert yp_{K} \Vert_\Psi < + \infty$ since $y = yq$ and $qp_K \in \cn_\Psi(M)$. 
Thus the ultraweakly closed convex set $\cC p_K \subset M$ is bounded both in the operator norm and in the $\Vert \cdot \Vert_\Psi$-norm and we can apply Lemma \ref{HI15} to it. In particular we can find $z \in \cC$, such that $zp_K$ is the unique element of $\cC p_K$ with minimal $\Vert \cdot \Vert_\Psi$-norm\footnote{Note however that $z$ itself needs not be unique.}. 

Let us check that such an element $z \in N$ satisfies $zp_K \neq 0$. First, since $qp_K \in \cn_\Psi(M)$, $\Psi(qp_K \cdot p_Kq)$ is a bounded normal linear functional on $M$, which is constant on $u_g^*\cC_0$. Hence it is constant on $u_g^*\cC$. It follows:
\[\Psi(qp_{K}u_g^*zp_{K}q) = \Psi(qp_{K}u_g^*yp_{K}q) = \Psi(p_{K}u_g^*(\sigma_g(q)yq)p_{K}) = \Psi(p_{K}u_g^*yp_{K}) \neq 0.\]
So we indeed find that $zp_K \neq 0$.

Note also that $\cC p_K$ is globally invariant under the affine action $\cU(qQ^Kq) \actson \sigma(q)Mq$ given by $u \cdot X := \sigma_g(u)Xu^*$. Since $\Psi$ centralizes $Q$, this action is $\Vert \cdot \Vert_\Psi$-isometric, so it fixes $zp_K$. Equivalently, we have $u_g^*zp_K \in (qQ^Kq)' \cap M$. Moreover $u_g^*zp_K$ belongs to $qMq$. By Proposition \ref{relcomK}, we have that $(qQ^Kq)' \cap qMq \subset qLK$, so $u_g^*zp_K \in qLKp_K$. Since the projection $p_K$ is both central and minimal inside $LK$, there exists a non-zero scalar $\lambda \in \C$ such that 
\begin{equation}\label{partial goal}u_g^*zp_K = \lambda qp_K.\end{equation}
We obtain that $u_gqp_K = (z/\lambda)p_K \in Np_K$. Unfortunately we don't have, a priori, any control on how small $\vert \lambda \vert$ is, so $z/\lambda$ could have a very large operator norm. To get around this issue, we would like to identify polar parts, but we need extra commutations properties. We apply the convex combination argument a second time.

Arguing as above, one can find an element $z_1 \in \cC'$, where $\cC' \subset \cC$ is the ultraweak closure of $\conv(\{ \sigma_g(u)zu^* \, , \, u \in \cU(qQ^Kq)\}) \subset N$, such that $z_1^*zp_K$ is the unique element in $(\cC')^* zp_K$ with minimal $\Vert \cdot \Vert_\Psi$-norm\footnote{Note that $(\cC')^* zp_K$ is nothing but the ultraweakly closed convex hull of $\{uz^*zp_Ku^* \, , \, u \in \cU(qQ^Kq)\}$.}. Then $z_1$ enjoys the following properties:
\begin{itemize}
\item $z_1 \in N$ and $z_1p_K = zp_K$, because for all $u \in \cU(qQ^Kq)$, $\sigma_g(u)zu^*p_K = \sigma_g(u)zp_Ku^* = zp_K$. In particular $z_1p_K \neq 0$;
\item $z_1^*zp_K \in (qQ^Kq)' \cap qMq$, by uniqueness of a $\Vert \cdot \Vert_\Psi$-norm minimizer inside $\cC' zp_K$. Hence Proposition \ref{relcomK} gives that $z_1^*zp_K = \lambda' qp_K$ for some non-zero scalar $\lambda' \in \C$.
\end{itemize}
In particular the above facts give $p_Kz_1^*z_1p_K = p_Kz_1^*zp_K = \lambda' qp_K$ and $\lambda'$ follows positive. Hence the equality $z_1^*z_1p_K = \lambda' qp_K = (\lambda' qp_K)^* = p_Kz_1^*z_1$ shows that $z_1^*z_1$ commutes with $p_K$. Write the polar decomposition $z_1 = u\vert z_1 \vert$, with $u \in N$ a partial isometry and $\vert z_1 \vert = (z_1^*z_1)^{1/2}$.
Note that $uqp_K$ is a partial isometry since $qp_K$ commutes with $z_1^*z_1$ (recall that $z_1 = z_1q$).

We have: $z_1p_K = u\vert z_1 \vert p_K = \sqrt{\lambda'} uqp_K$. Combining this with \eqref{partial goal}, we get
\[\lambda u_g qp_K = z p_K = z_1 p_K = \sqrt{\lambda'} uqp_K.\]
Hence $u_g qp_K$ and $u qp_K$ are proportional partial isometries; they have to coincide,
\[u_g qp_K = (uq) p_K \in (N)_1 p_K.\]
This proves Step 1.

{\bf Step 2.} For all compact open subgroups $K < K_0$ and $g \in G$, if there exists $x \in N$ such that $p_KE_K(u_g^*x)p_K \neq 0$, then $u_gp_K \in (N)_1p_K$.

Fix a subgroup $K < K_0$, $g \in G$ and $x \in N$ such that $p_KE_K(u_g^*x)p_K \neq 0$. 

Since $K$ acts minimally on $Q$ and $Q^K$ is with expectation inside $Q$, \cite[Lemma 5.1]{BHV15} implies that the trace $\Tr$ is still semi-finite on $Q^K$. Denote by $(q_i)_{i \in I}$ an increasing net of projections in $Q^K$ with finite trace that converges to $1$. 
Then there exists $i_0 \in I$ such that for all $i \geq i_0$, we have that $q_ip_KE_{K}(u_g^*x)p_{K}q_i \neq 0$.

By Step 1, we deduce that for all $i \geq i_0$, there exists $y_i \in (N)_1$ such that 
\begin{equation}\label{1112} u_gq_ip_K = y_i p_K.\end{equation}
Denote by $y \in (N)_1$ an ultraweak limit of the net $(y_i)_{i \in I}$. Then taking the corresponding limit in \eqref{1112} gives $u_gp_K = yp_K$, as desired.

{\bf Step 3.} The support of $x_0$ is contained in $H$. 

Take $g \in \supp(x_0)$. Fix a compact open subgroup $K < K_0$. Then $E_{K}(u_g^*x_0) \neq 0$. As the net of projections $(p_{L})_{L < G}$, converges strongly to the identity when $L \searrow \{e\}$, we have $p_L E_{K}(u_g^*x_0)p_{L} \neq 0$ for all small enough compact open subgroups $L$.

Take a normal open subgroup $L < K$, so that $p_L$ commutes with $u_g$ for all $g \in K$. By Lemma \ref{compactKdec}, we can write
\[p_LE_{K}(u_g^*x_0)p_{L}  = \sum_{h \in \lift(K/L)} p_Lu_hE_{L}(u_h^*E_{K}(u_g^*x_0))p_{L} = \sum_{h \in \lift(K/L)} u_hp_LE_{L}(u_{gh}^*x_0)p_{L}.\]
For all small enough compact open subgroups $L < G$ that is normal inside $K$, since the above sum is non-zero, we can find $h_{L} \in K$ such that $p_LE_{L}(u_{gh_{L}}^*x_0)p_{L} \neq 0$. Applying Step $2$ to $L$, $gh_{L}$ and $x_0$, we find an element $z_{L} \in (N)_1$ such that $u_{gh_{L}}p_{L} = z_{L}p_{L}$.

Note that there exists a net $(L_i)_i$ of compact open subgroups of $G$ which are all normal in $K$ and form a neighborhood basis of $e \in G$. This comes from the fact that any open subgroup $L < K$ has finite index inside $K$, so that $\cap_{h \in K} hLh^{-1}$ is an open normal subgroup of $K$, contained in $L$. By compactness of $K$ and $(N)_1$ (for the weak operator topology), there exist subnets of $(h_{L_i})_{i}$ and $(z_{L_i})_i$ that converge to elements $h \in K$ and $z \in (N)_1$, respectively. Taking ultraweak limits, we get that 
\[u_{gh} = \lim_i u_{gh_{L_i}}p_{L_i} = \lim_i z_{L_i}p_{L_i} = z \in N,\]
because the net of projections $(p_{L})_L$ converges ultrastrongly to the identity.
Hence $gh \in H$, and we conclude that $gK \cap H \neq \emptyset$.

As $K$ can be arbitrarily small and $H$ is closed, we conclude that $g \in H$. This finishes the proof of Step 3. 
Now the theorem follows from Theorem \ref{sgpthm}.
\end{proof}

\subsection{Examples of actions}
\label{section:examples}

\begin{prop}\label{Bernoulli}
Let $G$ be a totally disconnected group with a compact open subgroup $K < G$ such that $\cap_{g \in G} gKg^{-1} = \{e\}$ (let us call such a subgroup $K$ \emph{eventually malnormal}). Take a ${\rm II}_1$-factor $Q_0$, with trace $\tau_0$. The following $G$-actions satisfy the assumptions of Theorem \ref{intermediate subfactor}, so they are strictly outer and satisfy the Intermediate Subfactor Property.
\begin{enumerate}
\item The Bernoulli action $G \actson (Q_0,\tau_0)^{\ovt G/K}$ obtained by shifting indices; 
\item The free Bernoulli action $G \actson \ast_{G/K} (Q_0,\tau_0)$;
\end{enumerate}
\end{prop}
\begin{proof} Let us first check separately the minimality condition on the $K$-action for each case.

(1) Put $Q = Q_0^{\ovt G/K}$. Since $K$ is eventually malnormal inside $G$, it acts faithfully on $G/K$ and hence on $Q$. Since $K$ is compact open in $G$ it is commensurated, and hence it acts on $G/K$ with finite orbits. Let us denote by $\cO_j$, $j \in J$ these orbits and by $Q_j := Q_0^{\ovt \cO_j}$. Example \ref{irreducible} shows that for all $j$, the fixed point algebra $Q_j^{K} \subset Q_j$ is an irreducible subfactor. Since $Q^K$ contains $\ovt_{j \in J} Q_j^K$, it is irreducible inside $Q = \ovt_{j \in J} Q_j $. So $K$ acts minimally.

(2) Put $Q = \ast_{G/K} Q_0$. Again, since $K$ is not normal inside $G$, it acts faithfully on $Q$. Because of the free product situation, it is clear that the copy of $Q_0$ located at the label $K \in G/K$ is irreducible inside $Q$. Moreover this algebra is contained in $Q^K$, so $K$ acts faithfully.

We now check the relatively properly outer condition simultaneously for both situations. Take $g \in G \setminus K$ and decompose $Q$ as a product 
\[Q \simeq Q_0^{(K)} \ovt Q_0^{(gK)} \ovt P \qquad (\text{or } Q \simeq Q_0^{(K)} \ast Q_0^{(gK)} \ast P),\]
where $Q_0^{(K)}$ and $Q_0^{(gK)}$ are the copies of $Q_0$ in position $K \in G/K$ and $gK \in G/K$ and $P$ is the tensor product (or free product) of all the remaining copies of $Q_0$. 

In both the tensor situation and the free situation, one easily checks that for all nets $(u_n)_n \subset \cU(Q_0^{(K)})$ and $(v_n)_n \subset \cU(Q_0^{(gK)})$ that converge weakly to $0$, and for all $a,b \in Q$, one has
\[\lim_n \tau(u_nav_nb) = 0, \text{ where $\tau$ is the trace of $Q$.}\]
Assuming that there exists $a \in Q_1$ such that $\sigma_g(x)a = ax \text{ for all } x \in Q_1^K$, we take for $(u_n)_n \subset \cU(Q_0^{(K)}) \subset \cU(Q^K)$ any net of unitaries that converges weakly to $0$, and we set $v_n = \sigma_g(u_n)$. We get
\[\Vert a \Vert_2 = \lim_n \tau(u_n^*a^*au_n) = \lim_n \tau(u_na^*v_na) = 0.\]
Hence $a = 0$, as desired.
\end{proof}

More generally, one can easily check the relative outerness condition in Theorem \ref{intermediate subfactor} when the action has a large commutant thanks to the following fact.

\begin{lem}\label{centlem}
Consider an action $G \actson Q$ and a closed subgroup $K < G$ whose action is minimal.
Assume that the centralizer $\Gamma$ of $G$ in $\Aut(Q)$ satisfies: for all $x \in Q \setminus \C$, there exists $\gamma \in \Gamma$ such that $\gamma(x) \notin \C x$.

Then the action is properly outer relative to $K$ if and only if the only elements of $G$ acting trivially on $Q^K$ are the elements of $K$.
\end{lem}
\begin{proof}
The only if part is trivial. Conversely assume that $K$ is precisely the set of elements of $G$ that act trivially on $Q^K$. Take $g \in G$ such that there exists a non-zero $a \in Q$ satisfying $\sigma_g(x)a = ax$ for all $x \in Q^K$. Since $K$ acts minimally, we can assume that $a$ is a unitary. For all automorphism $\gamma \in \Gamma$, we have
\[\sigma_g(\gamma(x))\gamma(a) = \gamma(a) \gamma(x), \quad \text{for all } x \in Q^K.\]
Moreover, using again that $\gamma$ commutes with the $G$-action, we see that $\gamma(Q^K) = Q^K$, and the previous equation reads as
\[\sigma_g(x)\gamma(a) = \gamma(a)x, \quad \text{for all } x \in Q^K.\]
In particular, we find that $a^*\gamma(a) \in (Q^K)' \cap Q = \C$. Hence for all $\gamma \in \Gamma$, we have $\gamma(a) \in \C a$. By our assumption on $\Gamma$, this leads to $a \in \C$, and hence $g$ fixes $Q^K$ pointwise. Thus $g \in K$ by assumption.
\end{proof}

Note that the condition on the centralizer $\Gamma$ in the above lemma is fulfilled as soon as $\Gamma$ admits a subgroup $\Gamma_0$ that preserves a state on $Q$, and such that $Q^{\Gamma_0} = \C$. Moreover, if $Q$ is a ${\rm II}_1$-factor the trace is invariant under any automorphism, and hence only the second condition needs to be verified.

We deduce the following result in the spirit of Vaes' examples \cite[Theorem 5.1]{V05}

\begin{cor}\label{diagac}
Fix a totally disconnected group and a compact open subgroup $K < G$.
Consider any faithful action $G \actson Q_0$ on a ${\rm II}_1$-factor such that 
\[\{g \in G \, \vert \, g_{|Q^K} = \id\} = K.\]

Then the diagonal action $G \actson (Q,\tau) := (Q_0,\tau_0)^{\ovt \N}$ satisfies the assumptions of Theorem \ref{intermediate subfactor}.
\end{cor}
\begin{proof}
The $K$-action on $Q_0$ being faithful, \cite[Theorem 5.1]{V05} implies that it is strictly outer, and hence minimal. Moreover the centralizer $\Gamma$ of such a diagonal action contains all shift automorphisms obtained by permuting indices. Hence the condition on $\Gamma$ appearing in Lemma \ref{centlem} is satisfied (see the comment after Lemma \ref{centlem}).

Thus the result follows from Lemma \ref{centlem}.
\end{proof}

\begin{rem}
Note that the Bernoulli shift action $G \actson Q_0^{\ovt G/K}$ as in Lemma \ref{Bernoulli} is sometimes a special case of diagonal action as in Corollary \ref{diagac}. For instance this happens if $Q_0$ is the hyperfinite factor. However this is not clear why this should be the case when $Q_0$ is a prime factor. Hence even the strict outerness for such actions does not follow from \cite[Theorem 5.1]{V05}.
\end{rem}

Before moving to the next section, let us briefly explain how to adapt our argument to cover some actions on type ${\rm III}$ factors.

\begin{rem}\label{large centralizer}
Let $G$ be a totally disconnected group and let $K<G$ be an eventually malnormal compact open subgroup, see Lemma \ref{Bernoulli}. Let $Q_0$ be an arbitrary diffuse factor admitting a faithful normal state $\phi_0$ with large centralizer (meaning that $(Q_0^{\phi_0})' \cap Q_0 = \C$). Then the Bernoulli shift $G \actson (Q_0,\phi_0)^{\ovt G/K}$ satisfies the Intermediate Subfactor Property.

Let us briefly explain. Denote by $(Q,\phi) := (Q_0,\phi_0)^{\ovt G/K}$. Then the centralizer $Q^\phi$ of $\phi$ is irreducible inside $Q$, and it is invariant under the $G$-action. One can show that for all compact open subgroups $L < G$ which is a finite intersection of conjugates of $K$, the fixed point algebra $Q_L := (Q^\phi)^L$ satisfies $Q_L' \cap (Q \rtimes G) \subset L(L)$. This is done by following the proof of Lemma \ref{relcomK1}, and by noting that the action is properly outer relative to $L$.

Therefore one can use the averaging argument from the proof of Theorem \ref{intermediate subfactor} but instead we average with elements in $\cU((Q^\phi)^L)$ for small groups $L$ as above. Note moreover that in this case since $\phi$ is a state (and not a weight), one does not need to bother with the projection $q$ appearing in Step 1 of the proof of Theorem \ref{intermediate subfactor}.
\end{rem}

In fact the above remark also applies for free Bernoulli actions with weaker assumptions on $(Q_0,\phi_0)$, but we will not elaborate on this.

\section{Existence of conditional expectations and operator valued weights}\label{sectionCEOVW}

In this Section we discuss various results about existence (or non-existence) of conditional expectations and operator valued weights in connection with Izumi, Longo, and Popa's paper \cite{ILP98}.

Let us start our discussion by investigating the existence of conditional expectations/operator valued weights for pairs of the form $Q \rtimes H \subset Q \rtimes G$ associated with closed subgroups $H < G$.

\subsection{Proof of Theorem \ref{ovwsg}}

Let us start with two lemmas which rely on the notion of support.

\begin{lem}\label{diagcomm}
Consider two actions $G \actson Q$ and $G \actson P$ and the diagonal action of $G$ on $Q \ovt P$. If the action $G \actson Q$ is strictly outer, then
\[Q' \cap ((Q \ovt P) \rtimes G) = P.\]
\end{lem}
\begin{proof}
Embed $G$ into $G \times G$ diagonally. Then $M := (Q \ovt P) \rtimes G$ is identified with a subalgebra of $\tilde{M} := (Q \ovt P) \rtimes (G \times G)$.  Note that
\[Q' \cap \tilde{M} \simeq \Big(Q' \cap (Q \rtimes G)\Big) \ovt (P \rtimes G) = \C \ot (P \rtimes G).\]
In particular any element $x \in Q' \cap M$, viewed as an element of $\tilde{M}$ has its support contained in the diagonal subgroup of $G \times G$ (because $x \in M$) and in $\{e\} \times G$ (because $x \in Q' \cap \tilde{M}$).
So such an element $x$ has its support contained in the trivial group. By Corollary \ref{Beurling}, we get $x \in Q' \cap (Q \ovt P) = P$, as wanted.
\end{proof}

\begin{lem}\label{posupp}
Consider an action $G \actson Q$ on an arbitrary von Neumann algebra and put $M = Q \rtimes G$. Take a weight $\Phi \in \cP(Q)$ and denote by $\Psi \in \cP(M)$ the corresponding dual weight.

For any non-zero $x \in \cn_\Psi(M)$, the support $\supp(x)$ has positive Haar measure inside $G$.
\end{lem}
\begin{proof}
We will use the description of the dual weight relying on the Hilbert algebra approach, \cite{Ha78a}. According to this approach, given the weight $\Phi$, there exists a left Hilbert algebra $A_\Phi$ satisfying the following properties.
\begin{itemize}
\item  The Hilbert completion of $A_\Phi$ is isomorphic to $L^2(G,\cH)$ and the left von Neumann algebra $L(A_\Phi)$ is identified with $Q \rtimes G$;
\item The dual weight $\Psi$ corresponds to the canonical weight on $L(A_\Phi)$ associated with the Hilbert algebra $A_\Phi$; 
\item For any $\Phi$-right bounded vector $\eta \in \cH$ and any $f \in \cK(G)$, the function $f\eta \in \cK(G,\cH)$ defined by $f\eta : g \mapsto f(g)\eta$ is a right bounded vector for $A_\Phi$. The corresponding operator $\pi_r(f\eta)$ is given by
\begin{equation}\label{riac} \pi_r(f\eta)\xi (g) = \int_G \Big( \pi_r(\eta)\sigma_s(\xi(gs))\Big) f(s^{-1}) \ds, \quad \text{for all } \xi \in L^2(G,\cH), \, g \in G,\end{equation}
where $\pi_r(\eta)$ denotes the operator on $\cH = L^2(Q,\Phi)$ associated with the right bounded vector $\eta$\footnote{Here we follow \cite{Ha78b} and use the uniqueness of the standard form to identify canonically $\cH$ with $L^2(Q,\Phi)$.}.
\end{itemize}
Now take a non-zero $x \in \cn_\Psi(M)$. By the above facts and \cite[Chapter VII.2]{Ta03}, there exists a non-zero left bounded vector $\xi \in L^2(G,\cH)$ such that $x = \pi_\ell(\xi)$ (that is, $x$ is the operator extending the left multiplication by $\xi$).
We claim that the support of $\xi$ as a function in $L^2(G,\cH)$ is contained in the support of $x$. In fact the equality holds, but we clearly only need this inclusion to deduce the lemma.

Take $g \in G$ in the function support of $\xi$. Take an open neighborhood $\Omega \subset G$ of the identity element $e$. We have to show that $P(g\Omega)\pi_\ell(\xi)P(\Omega)$ is non-zero. For all $f \in \cK(\Omega)$ and all $\Phi$-right bounded vector $\eta \in \cH$, we have
\[P(g\Omega)\pi_\ell(\xi)P(\Omega)(f\eta) = P(g\Omega)\pi_\ell(\xi)(f\eta) = P(g\Omega)\pi_r(f\eta)\xi.\]
We now check that for a suitable choice of $f$ and $\eta$ the above quantity is non-zero.

Since $g$ is in the function support of $\xi$, there exists $\eta_0 \in \cH$ such that
$g\Omega \cap \{h \in G \, \vert \, \langle \xi(h),\eta_0 \rangle \neq 0\}$ has positive measure. Since the set $\{ \pi_r(\eta_1)^*\eta_2 \, , \, \eta_1,\eta_2 \in \cH, \Phi \text{-right bounded}\}$ is dense inside $\cH$, we may find two $\Phi$-right bounded vectors $\eta_1, \eta_2 \in \cH$ such that $\pi_r(\eta_1)^*\eta_2$ is sufficiently close to $\eta_0$ so that the set $g\Omega \cap \{h \in G \, \vert \, \langle \xi(h),\pi_r(\eta_1)^*\eta_2 \rangle \neq 0\}$ has positive Haar measure. In particular, the function $\zeta \in L^2(G,\cH)$ defined by $h \mapsto \pi_r(\eta_1)\xi(h)$ satisfies $P(g\Omega)\zeta \neq 0$. Put $C := \Vert P(g\Omega)\zeta \Vert_2 > 0$.

For $f \in \cK(G)$, denote by $\check{f}$ the function $h \mapsto f(h^{-1})$. For all non-negative function $f \in \cK(G)$ such that $\int_G f = 1$ we have, by \eqref{riac}, 
\begin{align*}
\Vert \pi_r(\check{f}\eta_1)\xi - \zeta\Vert_2^2 & = \int_G \Vert (\pi_r(\check{f}\eta_1)\xi)(h) - \zeta(h)\Vert^2 \dh\\
& = \int_G \Vert \int_G f(s)\pi_r(\eta_1)\Big(\sigma_s(\xi(hs)) - \xi(h)\Big)\ds \Vert^2 \dh\\
& \leq \int_G \Big(\int_G f(s)\Vert \pi_r(\eta_1)\Vert \Vert \sigma_s(\xi(hs)) - \xi(h) \Vert \ds \Big)^2 \dh\\
& = \Vert \pi_r(\eta_1)\Vert^2 \int_{G \times G \times G} f(s)f(t) \Vert \sigma_s(\xi(hs)) - \xi(h) \Vert \Vert \sigma_t(\xi(ht)) - \xi(h) \Vert \ds\dt \dh\\
& \leq \Vert \pi_r(\eta_1)\Vert^2 \int_{G \times G} f(s)f(t) \Vert \delta_G(s)^{-1/2}\rho_G(s)(\xi) - \xi \Vert \Vert \delta_G(t)^{-1/2}\rho_G(t)(\xi) - \xi \Vert \ds\dt\\
& = \Big(\Vert \pi_r(\eta_1)\Vert \int_{G} f(s) \Vert \delta_G(s)^{-1/2}\rho_G(s)(\xi) - \xi \Vert \ds\Big)^2,
\end{align*}
where $\rho_G$ is the representation defined in \eqref{rhoG}. Since $\rho_G$ is a continuous representation and $\delta_G$ is a continuous function on $G$, we get that if $f$ is supported on a small enough neighborhood of $e$, then $\Vert \pi_r(\check{f}\eta_1)\xi - \zeta\Vert_2 < C/2$. By definition of $C$ we get:
\[ \Vert P(g\Omega)\pi_r(\check{f}\eta_1)\xi \Vert \geq \Vert P(g\Omega)\zeta \Vert - \Vert \pi_r(\check{f}\eta_1)\xi - \zeta \Vert \geq C/2.\]
So there indeed exists a $\Phi$-right bounded vector $\eta_1$ and a function $f_1 = \check{f}$ which is supported on $\Omega$ such that
\[P(g\Omega)xP(\Omega)(f_1\eta_1) = P(g\Omega)\pi_\ell(\xi)(f_1\eta_1) \neq 0.\qedhere\]
\end{proof}

\begin{proof}[Proof of Theorem \ref{ovwsg}]
We prove the two facts separately.

(1) First assume that the modular functions $\delta_G$ and $\delta_H$ coincide on $H$. Then by \cite[Theorem 3.2]{Ha78a}, for any weight $\Phi \in \cP(Q)$, the dual weight $\Psi_G \in \cP(Q \rtimes G)$ and $\Psi_H \in \cP(Q \rtimes H)$ satisfy
\[\sigma_t^{\Psi_G}(x) = \sigma_t^{\Psi_H}(x), \quad \text{for all } x \in Q \rtimes H, t \in \R.\]
Therefore there exists a nfs operator valued weight $T \in \cP(Q \rtimes G, Q \rtimes H)$ by \cite[Theorem 5.1]{Ha79b}.

Conversely, assume that $\cP(Q \rtimes G,Q \rtimes H)$ is non-trivial. By \cite[Theorem 5.9]{Ha79b}, and Remark \ref{rkcomm} we deduce that there exists an operator valued weight $T \in \cP(\tilde M,M)$, where $M = Q \rtimes G$ and $\tilde M = (\cL^\infty(G/H) \ovt Q) \rtimes G$ (because $\tilde{M}$ is isomorphic to the basic construction of $Q \rtimes H \subset M$).

Our intermediate goal is to deduce that there exists a $G$-invariant nfs weight on $A := \cL^\infty(G/H)$. Unfortunately, we don't know a priori that $T$ is semi-finite on $A$. To get around this issue we will exploit the fact that the action is strictly outer and use modular theory.

Let us consider the following operator valued weights in $\cP(\tilde M,Q)$.
\begin{itemize}
\item $T_1 := T_G \circ T$, where $T_G \in \cP(Q \rtimes G,Q)$ is the Plancherel operator valued weight.
\item $T_\phi := (\phi \ot \id) \circ \tilde{T}_G$, where $\tilde{T}_G \in \cP(\tilde{M},A \ovt Q)$ is the Plancherel operator valued weight and $\phi \ot \id \in \cP(A \ovt Q,Q)$ is the tensor product operator valued weight associated to some weight $\phi \in \cP(A)$ and the identity map on $Q$ (see \cite[Theorem 5.5]{Ha79b}).
\end{itemize}
Fix $\phi \in \cP(A)$. By Lemma \ref{diagcomm}, we have that $Q' \cap \tilde{M} = A$, hence the Connes Radon-Nikodym cocycle $(DT_1:DT_\phi)_t$ in the sense of \cite[Definition 6.2]{Ha79b} takes values into $A$, and it is a $1$-cocycle for the flow $\sigma_t^{T_\phi}$.
By construction, for any weight $\psi \in \cP(Q)$, the weight $\psi \circ T_\phi$ is simply the dual weight associated to $\phi \ot \psi \in \cP(A \ovt Q)$. Hence $(\sigma_t^{T_\phi})_{|A} = \sigma_t^{\phi} = \id$, since $A$ is abelian. We conclude that $(u_t)_t := (DT_1:DT_\phi)_t$ is a one parameter subgroup of unitaries of $A$.

By \cite[Th\'eor\`eme 1.2.4]{Co73}, there exists a nfs weight $\phi' \in \cP(A)$ such that $(D\phi':D\phi)_t = u_t$ for all $t \in \R$.

{\bf Claim 1.} The weight $\phi'$ is $G$-invariant.
We will denote generically by the letter $\sigma$ all the $G$ actions. We fix $g \in G$ and show that the Connes Radon-Nikodym derivative $(D\phi' \circ \sigma_g:D\phi')_t$ is equal to $1$ for all $t \in \R$. We have
\begin{equation}\label{inv1} 
(D\phi' \circ \sigma_g:D\phi')_t = (D\phi' \circ \sigma_g:D\phi \circ \sigma_g)_t(D\phi \circ \sigma_g:D\phi)_t(D\phi:D\phi')_t.
\end{equation}
By Lemma \ref{RNautom}, we have that $(D\phi' \circ \sigma_g:D\phi \circ \sigma_g)_t = \sigma_g^{-1}((D\phi':D\phi)_t) = \sigma_g^{-1}(u_t)$. Hence \eqref{inv1} becomes
\begin{equation}\label{inv2} 
(D\phi' \circ \sigma_g:D\phi')_t = \sigma_g^{-1}(u_t)(D\phi \circ \sigma_g:D\phi)_tu_t.
\end{equation}

We will show that the right hand side above is equal to $1$ by computing all the terms in the equality
\begin{equation}\label{inv3}
(DT_1^g : DT_1)_t = (DT_1^g:DT_\phi^g)_t(DT_\phi^g:DT_\phi)_t(DT_\phi:DT_1)_t.
\end{equation}
Here we denoted by $T_1^g$ (resp. $T_\phi^g$) the operator valued weight $\sigma_g^{-1} \circ T_1 \circ \Ad(u_g) \in \cP(\tilde{M},Q)$ (resp. $\sigma_g^{-1} \circ T_\phi \circ \Ad(u_g)$).

Take a weight $\psi \in \cP(Q)$. By definition of the Connes Radon-Nikodym derivative for operator valued weights, we have
\begin{align*}
(DT_1^g:DT_\phi^g)_t & = (D(\psi \circ T_1^g):D(\psi \circ T_\phi^g))_t\\
& = (D(\psi \circ \sigma_g^{-1} \circ T_1) \circ \Ad(u_g):D(\psi \circ \sigma_g^{-1} \circ T_\phi) \circ \Ad(u_g))_t\\
& = \sigma_g^{-1}((D(\psi \circ \sigma_g^{-1} \circ T_1):D(\psi \circ \sigma_g^{-1} \circ T_\phi))_t)\\
& = \sigma_g^{-1}((DT_1:DT_\phi)_t),
\end{align*}
where the third equality follows from Lemma \ref{RNautom}.
Hence \eqref{inv3} becomes
\begin{equation}\label{inv4}
(DT_1^g : DT_1)_t = \sigma_g^{-1}(u_t)(DT_\phi^g:DT_\phi)_tu_t^*.
\end{equation}
Now \cite[Theorem 3.1]{Ha78b} is easily seen to imply that $T_1^g = \delta_G(g)T_1$ while $T_\phi^g = \delta_G(g)T_{\phi \circ \sigma_g}$. So $(DT_1^g : DT_1)_t = \delta_G(g)^{it}$ and
\[(DT_\phi^g:DT_\phi)_t  = \delta_G(g)^{it} (DT_{\phi \circ \sigma_g}:DT_\phi)_t = \delta_G(g)^{it} (D\phi \circ \sigma_g:D\phi)_t\]
Altogether, \eqref{inv4} can be rewritten
\begin{equation}\label{inv5}
\delta_G(g)^{it} = \sigma_g^{-1}(u_t)\delta_G(g)^{it} (D\phi \circ \sigma_g:D\phi)_tu_t^*.
\end{equation}
And we see that the right hand side of \eqref{inv2} is equal to $1$, proving Claim 1.

Recall that $A = \cL^\infty(G/H)$ is the subalgebra of right $H$-invariant functions inside $L^\infty(G)$. Denote by $q: G \to G/H$ the quotient map. The formula $B \subset G/H \mapsto \phi'(\mathbf 1_{q^{-1}(B)})$ defines a measure $\nu$ on the Borel $\sigma$-algebra of $G/H$. This measure is $G$-invariant and non-zero (because $\phi'$ is faithful). 

{\bf Claim 2.} The $G$-invariant measure $\nu$ is finite on every compact set of $G/H$.

Since $\phi'$ is semi-finite (and faithful), there exists a Borel set $B \subset G/H$ such that $0 < \nu(B) < \infty$. Take a compactly supported non-negative function on $G$, $f \in \cK(G)$, $f \neq 0$. Then the function $f \ast 1_B$ defined as follows is continuous:
\[(f \ast 1_B)(xH) := \int_G f(g)1_B(g^{-1}xH)\dm_G(g), \quad xH \in G/H.\]
By Fubini-Tonelli's Theorem, we have the key equation
\begin{align*}
\int_{G/H} (f\ast 1_B)\dnu & = \int_{G/H}\int_G f(g)1_B(g^{-1}xH)\dm_G(g)\dnu(xH)\\
& = \int_G f(g) \nu(gB)\dm_G(g)\\
& = \nu(B)\int_G f\dm_G.
\end{align*}
This key equation tells us first that the continuous non-negative function $f \ast 1_B$ is non-zero. So there exists an open set $U \subset G/H$ and $\alpha > 0$ such that $1_U \leq \alpha (f \ast 1_B)$. The key equation also tells us that $f \ast 1_B$ is $\nu$-integrable, and in particular, $\nu(U) < \infty$. Since any compact set $K \subset G/H$ can be covered by finitely many translates of $U$, Claim 2 follows. 

By \cite[Corollary B.1.7]{BHV08} we deduce from the existence of $\nu$ that the modular functions $\delta_G$ and $\delta_H$ must coincide on $H$. This proves (1).

(2) As mentioned earlier in the paper, if $H$ is open inside $G$, then the indicator function $1_H$ is continuous and positive definite on $G$. Then the associated multiplier (see \cite[Theorem 3.1.a]{Ha78b}) gives the desired conditional expectation from $Q \rtimes G$ onto $Q \rtimes H$.

Conversely, assume that $H$ is not open inside $G$. If the modular functions $\delta_G$ and $\delta_H$ do not coincide on $H$, then part (1) ensures that there is no nfs operator valued weight from $Q \rtimes G$ onto $Q \rtimes H$, and in particular, no conditional expectation. 

Assume now that the modular functions do coincide. 
Fix a nfs weight $\Phi \in \cP(Q)$, and denote by $\Psi_H \in \cP(Q \rtimes H)$ and $\Psi_G \in \cP(Q \rtimes G)$ the associated dual weights. As we saw in the proof of (1), \cite[Theorem 5.1]{Ha79b} implies that there exists an operator valued weight $T \in \cP(Q \rtimes G,Q\rtimes H)$ such that $\Psi_H \circ T = \Psi_G$. As the inclusion $Q \rtimes H \subset Q \rtimes G$ is irreducible, it suffices to show that $T$ is unbounded, by \cite[Theorem 6.6]{Ha79b}.

Fix an element $f \in \cK(H,Q)$. Note that since $H$ is not open inside $G$ it has measure $0$ inside $G$\footnote{Indeed for any positive measure Borel set $A$ inside $G$, the product $A^{-1}A$ contains a neighborhood of the identity.}. In particular, since $\supp(f) \subset H$, Lemma \ref{posupp} implies that $f \notin \cn_{\Psi_G}(Q \rtimes G)$. In contrast, $f \in \cn_{\Psi_H}(Q \rtimes H)$, and hence the expression $\psi_f := \Psi_H(f^* \cdot f)$ defines a normal positive linear functional on $Q \rtimes H$. We have
\[T(1)(\psi_f) = \psi_f \circ T(1) = \Psi_H \circ T(f^*f) = \Psi_G(f^*f) = + \infty.\]
Hence $T(1) \in \widehat{Q \rtimes H}_+ \setminus Q \rtimes H$, and so $T$ is not bounded. The proof is complete.
\end{proof}

\subsection{When conditional expectations do exist. Applications to Hecke pairs}

Before proving Theorem \ref{theo:Heckepair}, let us mention that the argument of Choda applies beyond the setting of discrete groups.

\begin{thm}[\cite{Ch78}, Theorem 3]\label{Chodalc}
Consider a strictly outer action $G \actson Q$ of a locally compact group and take a von Neumann subalgebra $N \subset Q \rtimes G$ that contains $Q$ and that is the range of a faithful normal conditional expectation $E_N: Q \rtimes G \to N$.

Then $N$ is of the form $Q \rtimes H$ for some open subgroup $H$ of $G$.
\end{thm}
\begin{proof}
As usual, consider the closed subgroup $H$ of $G$ defined by $H = \{g \in G \, \vert \, u_g \in N\}$, so that $Q \rtimes H \subset N$. Let us show that the converse inclusion also holds.

Since $N$ contains $Q$, we have $u_g^*E_N(u_g) \in Q' \cap (Q \rtimes G)$, for all $g \in G$. As the action is strictly outer, we deduce that for any $g \in G$, $E_N(u_g)$ is a scalar multiple of $u_g$. If this scalar multiple is non-zero, this means that $u_g \in N$, and hence the scalar in question must be $1$. So we obtain the following computation:
\[ E_N(au_g) = \mathbf{1}_{\{g \in H\} } au_g \in Q \rtimes H, \quad \text{ for all } a \in Q \text{ and } g \in G.\]
Since $E_N$ is normal we deduce by linearity and density that $E_N(Q \rtimes G) \subset Q \rtimes H$. Thus $N \subset Q \rtimes H$, and we have equality.
The fact that $H$ is open follows from Theorem \ref{ovwsg}.
\end{proof}

We now mention a lemma that provides existence of conditional expectations. It follows from the main technical result of \cite{ILP98}.

\begin{lem}\label{lem:E_L}
Consider a compact open subgroup $K<G$ in a locally compact group.
Let $\sigma:G\actson Q$ be a strictly outer action and put $M := Q \rtimes G, N:= Q\rtimes K.$
Then any intermediate \VNA\ $N\subset L \subset M$ is the range of a normal faithful conditional expectation $E_L:M \to L$.
\end{lem}
\begin{proof}
We show that $N \subset M$ satisfy the assumptions of \cite[Corollary 3.1]{ILP98}. Observe that the inclusion $N\subset M$ is irreducible since the action of $G$ is strictly outer.
As mentioned earlier in the paper, since $K$ is open inside $G$, there exists a conditional expectation $E = E_K$ from $M = Q \rtimes G$ onto $N = Q \rtimes K$.

By Remark \ref{rkcomm}, the basic construction $M_1$ of the $N\subset M$ is isomorphic to $(Q\ovt \ell^\infty(G/K))\rtimes G$, where $G$ acts diagonally. The embedding $M \subset M_1$ is given by $au_g \mapsto (a \ot 1)u_g$ for all $a \in Q$, $g \in G$.
Note that $G/K$ is discrete since $K$ is open.
By Lemma \ref{diagcomm}, we have $Q' \cap M_1 = \ell^\infty(G/K)$. Hence, the relative commutant $N^c:=N'\cap M_1$ is isomorphic to the \VNA\ of $K$-bi-invariant maps $\ell^\infty(K\backslash G/ K)$.

We represent faithfully $M_1$ on $\cH\ot \ell^2(G/K)\ot L^2(G)$ in the obvious way. In this picture, the Jones projection $e_N$ is the orthogonal projection onto $\cH\ot \C \delta_K\ot L^2(G)$ where $\delta_K$ is the Dirac mass of the coset $K\in G/K$.
Consider the dual \OVW\ $\widehat E\in\cP(M_1,M)$ of $E$. We have that $\widehat E(e_N)=1$ by \cite[Lemma 3.1]{Ko85}.
This implies that $\widehat E(1_{KgK})=[K : K\cap gKg^{-1}]$, where $1_{KgK}$ is the characteristic function of $KgK, g\in G$.
Since $K<G$ is a compact open subgroup, the index $[K : K\cap gKg^{-1}]$ is finite for any $g\in G$ and thus the \OVW\ $T:= E\circ\widehat E$ is semi-finite on $N^c$. Given any weight $\theta\in\cP(N)$, \cite[Theorem 6.6]{Ha79b} states that the restriction $\sigma^{ \theta \circ T}\vert_{N^c}$ of the modular flow associated to $\theta\circ T$ to $N^c$ is equal to the modular flow of the restriction of $\theta\circ T$ to $N^c$. Therefore, $\sigma^T_t(x)=x$ for any $t\in \R, x\in N^c$ since $N^c$ is commutative.
Hence, the pair $N \subset M$ indeed satisfies the assumptions of \cite[Corollary 3.11]{ILP98}, which implies the lemma.
\end{proof}

\begin{proof}[Proof of Theorem \ref{theo:Heckepair}]
This follows immediately by combining Theorem \ref{Chodalc} and Lemma \ref{lem:E_L}.
\end{proof}

Let us now derive applications to crossed-products by Hecke pairs of groups. 
By definition, a {\it Hecke pair} is a pair of groups $(G,H)$ such that $H$ is a subgroup of $G$ which is commensurated (or almost normal) in $G$, in the sense that $H^g:=H\cap g H g^{-1}$ has finite index in $H$ and $gHg^{-1}$ for all $g\in G$.
We refer to \cite{AD} and \cite{Pa, Br16} for more details on facts below.

A typical example of a Hecke pair $(G,H)$ arises when $G$ is a subgroup of the automorphism group of a locally finite connected graph $\Gamma$ and $H$ is the subgroup of elements of $G$ that stabilize a given vertex of $\Gamma$. In fact, this example is somewhat generic, see \cite[Theorem 2.15]{AD12}.

To any Hecke pair $(G,H)$ can be associated its Schlichting completion $(\wG,\wH)$, which is a new Hecke pair, for which $\wG$ is a totally disconnected group and $\wH$ is a compact open subgroup of $\wG$. The precise construction goes as follows: view $G$ as a subgroup of $S_{G/H}$, the permutation group of $G/H$. Endow $S_{G/H}$ with the topology of pointwise converge (where $G/H$ is viewed as a discrete space), and define $\wG$ (resp. $\wH$) to be the closure of $G$ (resp. $H$) inside $S_{G/H}$. The idea of using the Schlichting completion to study operator algebras of Hecke pairs goes back to Tzanev \cite{Tz01}. The key observation is Proposition \ref{prop:corner-hecke} below.

We will say that an action $G \actson Q$ is {\it action of the Hecke pair} $(G,H)$ if it extends continuously to an action of the Schlichting completion $\wG$. Let us define the crossed-product von Neumann algebra associated with such an action.
The construction generalizes that of C$^*$-algebras associated with Hecke pairs (which correspond to the trivial action $G \actson \C$), which originate in \cite{BC95}. We also refer to \cite{AD} for a general treatment. 
Crossed-product $C^*$-algebras for actions of Hecke pairs were then defined in \cite{Pa} and the von Neumann algebraic version in \cite[Section 4]{Br16}. 

Let $\CQGH$ be the space of continuous functions $f:G\rightarrow Q$ such that 
$f(hgk)=\sigma_h(f(g))$ for any $g\in G$, $h,k\in H$, and such that the induced functions $\overline f :G/H\rightarrow Q$ is finitely supported.
Note that if $f\in\CQGH$ and $g\in G$, then $f(g)$ is fixed by $\sigma(H\cap gHg^{-1})$.
We define a multiplication and an involution $*$ on $\CQGH$ as follows.
$$f_1f_2(g)=\sum_{s\in\GH}f_1(s)\sigma_s(f_2(s^{-1}g)), \text{ for any } f_1,f_2\in\CQGH \text{ and } g\in G,$$
where $\GH$ is a system of representatives of $G/H$,
$$f^*(g)=\sigma_g(f(g^{-1})^*), \text { for any } f\in\CQGH\ \text{ and } g\in G.$$
The space $\CQGH$ endowed with those operations is a unital $*$-algebra.
It contains a copy of the fixed point \VNA\ $Q^H$ via the map $j(a)(g)=\chi_H(g) a, a\in Q^H, g\in G,$ where $\chi_H$ is the characteristic function of $H$.

Assume that $Q$ is standardly represented on a Hilbert space $\cH$ and denote again by $\sigma: G \to \cU(\cH)$ the canonical implementation of the action, see \cite{Ha75}. Let $L^2(G/H,\cH)$ be the Hilbert space of $L^2$-functions from $G/H$ to $\cH$, where $G/H$ is viewed as a discrete space (and endowed with the counting measure). Consider the subspace
\[\cK:=\{ \xi \in L^2(G/H,\cH) \, \vert \, \xi(h g H) = \sigma_h( \xi ( gH ) ) , \forall h\in H, gH\in G/H \}.\]

A similar proof to that of \cite[Proposition 5.1]{Br16} gives us the following.

\begin{lem}
The map $\pi:\CQGH\rightarrow B(\cK)$ defined by
$$\pi(f)\xi ( gH )=\sum_{s\in\GH} f(s) \sigma_s ( \xi ( s^{-1} g H ) ), f\in\CQGH, \xi\in\cK, gH\in G/H,$$
is a bounded representation of the $*$-algebra $\CQGH$. We call it the {\it standard representation} of the Hecke algebra.
\end{lem}

Denote by $\VN[Q;G,H]$ the bicommutant of $\pi(\CQGH)$.
We call it the crossed-product \VNA\ of $Q$ by $(G,H)$.

The next proposition relates crossed-products by Hecke pairs to the group crossed-product by the Schlichting completion, it generalizes Lemma \ref{corner}. As mentioned above, this idea goes back to Tzanev \cite{Tz01}. 

\begin{prop}\label{prop:corner-hecke}
Consider an action of a Hecke pair $(G,H)$ on a von Neumann algebra $Q$. Denote by $(\wG,\wH)$ the Schlichting completion of $(G,H)$. Then the pairs 
\[\left(Q^H \subset \VN[Q;G,H]\right) \text{ and } \left(p_{\wH} (Q \rtimes \wH) p_{\wH} \subset p_{\wH} (Q \rtimes \wG) p_{\wH}\right)\]
are isomorphic. Here $p_{\wH}$ is the averaging projection associated with the compact open subgroup $\wH < \wG$ as defined above.
\end{prop}
\begin{proof}
Note that the coset spaces $G/H$ and $\wG/\wH$ are naturally isomorphic. Since the action $G \actson Q$ extends continuously to an action of $\wG$, the pairs 
\[\left( Q^{H} \subset \VN [ Q;G,H ] \right) \text{ and } \left(Q^{\wH}\subset\VN [ Q;\wG,\wH ] \right)\]
are isomorphic. Moreover, the Schlichting completion of $(\wG,\wH)$ is equal to itself. Replacing $(G,H)$ with $(\wG,\wH)$ if necessary, we may assume that $H$ is compact open inside $G$, and $\wG = G$, $\wH = H$.

Consider the isometry $u:\cK \to L^2(G,\cH)$ given by the formula 
\[u\xi(g) = \frac{1}{m_G(H)^{1/2}} \, \sigma_g^{-1}(\xi(gH)), \quad g\in G.\]
Put $M := Q\rtimes G$. Observe that the map $x\in p_HMp_H\mapsto u^* x u \in B(\cK)$ is a representation that is equivalent to $p_H M p_H$ acting on $p_H J p_H J L^2(G,\cH)$, where $J$ is the conjugation operator on the standard representation $L^2(G,\cH)$ of $M$. In particular, $x\in p_HMp_H \mapsto u^* x u \in B(\cK)$ is a faithful representation.

Consider the map $\phi: \CQGH \rightarrow \cK(G,Q)$ defined by the formula
\[\phi(f)(g)= \frac{1}{m_G(H)} \, \sigma_g^{-1}(f(g)), \quad \text{for any } f\in\CQGH, \, g\in G.\] 
Observe that the range of $\phi$ is the space of functions $F \in \cK(G,Q)$ such that $F(kgl)=\sigma_l^{-1}(F(g))$ for any $g\in G, k,l\in H$.
This range is precisely the corner $p_H\cK(G,Q) p_H$ and we have that 
\begin{equation}\label{equa:uphiu}
u^* \phi(f)u=f \text{ for any } f\in\CQGH.
\end{equation}
Therefore, $x\in p_HMp_H\mapsto u^* x u \in B(\cK)$ is an isomorphism of \VNA s onto $\VNQGH$. By Lemma \ref{corner}, we have that $p_H(Q \rtimes H)p_H = Q^Hp_H$ and one easily checks that $u^*(Q^Hp_H)u = Q^H$.
\end{proof}

Observe that if $H<L<G$ is an intermediate closed group, then there is a natural identification of $\C[Q;L,H]$ as a subalgebra of $\C[Q;G,H]$ of all functions that are supported on $L$. This identification extends to the respective crossed-product \VNA s.
To see this with little effort we can use the above proposition as follows. By definition of the Schlichting completion we have that $\tilde L$ is isomorphic to the closure of $L$ inside $\tilde G$, where $\tilde L$ and $\tilde G$ are the Schlichting completions of $L$ and $G$ with respect to the subgroup $H$.
As mentioned in Section \ref{sec:CP}, that the crossed-product $Q\rtimes\tilde L$ is isomorphic to the weak closure of the algebraic crossed-product $Q\rtimes_{\text{alg}} \tilde L$ inside $Q\rtimes \tilde G$.
Proposition \ref{prop:corner-hecke} and this later fact imply that there exists an injective morphism of \VNA s $\phi:\VN[Q;L,H]\rightarrow \VNQGH$ that sends $Q^H$ to itself and such that $\phi(\VN[Q;L,H])$ is the weak closure of $\C[Q;L,H]$ inside $\VNQGH$.

Theorem \ref{theo:Heckepair} together with Proposition \ref{prop:corner-hecke} implies the following result for crossed-products by Hecke pairs.

\begin{cor}\label{cor:Hecke}
Consider an action of a Hecke pair $(G,H)$ on a von Neumann algebra $Q$.
If $Q^H\subset L\subset \VNQGH$ is an intermediate \VNA, then there exists an intermediate group $H<H'<G$ such that $L$ is isomorphic to the crossed-product $\VN[Q;H',H]$, which we identified as a subalgebra of $\VNQGH$.
\end{cor}

Recently, it has been shown that Hecke pairs appear in subfactor theory.
Consider a finite index subfactor $N\subset M$ and its symmetric enveloping inclusion $T\subset S$, see \cite{Po94} (or \cite{LR95} in the type III finite depth setting).
Then in some cases, there exists a Hecke pair $(G,K)$ and actions $G\actson M\ovt M\op, K\times K\actson M\ovt M\op$ such that $T\subset S$ is isomorphic with $(M\ovt M\op)^{K\times K} \subset \VN[M\ovt M\op; G, K]$, see \cite[Theorem 5.5]{Br16}.
Hence, our last corollary gives us information about the lattice of intermediate subfactor in the symmetric enveloping inclusion of $N\subset M$.


\bibliographystyle{plain}

\end{document}